\documentclass{imsart}

\RequirePackage[OT1]{fontenc}
\usepackage{bbm}
\usepackage[numbers]{natbib}
\usepackage{amsthm,amsmath}
\usepackage{graphics,ifthen}
\usepackage{latexsym}
\usepackage{mathrsfs}
\usepackage{amssymb}
\startlocaldefs

\newtheoremstyle{example}{\topsep}{\topsep}%
     {}%         Body font
     {}%         Indent amount (empty = no indent, \parindent = para indent)
     {\rmfamily}% Thm head font
     {}%        Punctuation after thm head
     {\newline}%     Space after thm head (\newline = linebreak)
     {\thmname{#1}\thmnumber{ #2}\thmnote{ #3}}%         Thm head spec

   \theoremstyle{example}

\def\min{\text{min}}
\newcommand{\indic}{\mathbb{I}}

\numberwithin{equation}{section}
\theoremstyle{plain}
\newtheorem{thm}{Theorem}[section]
\newtheorem{defin}{Definition}[section]

\newtheorem{prop}{Proposition}[section]

\newtheorem{rem}{Remark}[section]

\newcommand{\Lower}[2]{\smash{\lower #1 \hbox{#2}}}
\newcommand{\ben}{\begin{enumerate}}
\newcommand{\een}{\end{enumerate}}
\newcommand{\bi}{\begin{itemize}}
\newcommand{\ei}{\end{itemize}}

\endlocaldefs

\begin{document}

\begin{frontmatter}
\title{Poisson Latent Feature Calculus for Generalized Indian Buffet Processes\protect} \runtitle{IBP}
%\thankstext{T1}{Footnote to the title with the `thankstext' command.}

\begin{aug}
\author{\fnms{Lancelot F.} \snm{James}\thanksref{t1}\ead[label=e1]{lancelot@ust.hk}}

\thankstext{t1}{Supported in
part by the grant RGC-HKUST 601712 of the HKSAR.}

\runauthor{Lancelot F. James}

\affiliation{Hong Kong University of Science and Technology}

\address[a]{Lancelot F. James\\ The Hong Kong University of Science and
Technology, \\Department of Information Systems, Business Statistics and Operations Management,\\
Clear Water Bay, Kowloon, Hong Kong.\\ \printead{e1}.}

\contributor{James, Lancelot F.}{Hong Kong University of Science
and Technology}

\end{aug}

\begin{abstract}
The purpose of this work is to describe a unified, and indeed simple,  mechanism for  non-parametric Bayesian analysis, construction and generative sampling of a large class of 
latent feature models which one can describe as generalized notions of Indian Buffet Processes~(IBP). This is done via the Poisson Process Calculus as it now relates to latent feature models. The IBP, first arising in a Bayesian Machine Learning context, was ingeniously devised by Griffiths and Ghahramani in (2005) and its generative scheme is cast in terms of customers entering sequentially an Indian Buffet restaurant and selecting previously sampled dishes as well as new dishes. In this metaphor dishes corresponds to latent features, attributes, preferences shared by individuals. The IBP, and its generalizations, represent an exciting class of models well suited to handle high dimensional statistical problems now common in this information age. In a survey article Griffiths and Ghahramani note applications for Choice models, modeling protein interactions, independent components analysis and sparse factor analysis, among others. 
 The IBP is based on the usage of conditionally independent Bernoulli random variables, coupled with completely random measures acting as Bayesian priors, that are used to create sparse binary matrices.  This Bayesian non-parametric view was a key insight due to Thibaux and Jordan (2007). One way to think of generalizations is to to use more general random variables. Of note in the current literature are models employing Poisson and Negative-Binomial random variables. 
However, unlike their closely related counterparts, generalized Chinese restaurant processes, the ability to analyze IBP models in a systematic and general manner is not yet available. The limitations are both in terms of knowledge about the effects of different priors and in terms of models based on a wider choice of random variables.  This work will not only provide a thorough description of the properties of existing models but also provide  a simple template to devise and analyze new models. We close by proposing and analyzing general classes of multivariate processes.
\end{abstract}



\end{frontmatter}
%\tableofcontents
\section{Introduction}
The purpose of this work is to describe a unified, and indeed simple,  mechanism for non-parametric Bayesian analysis of a large class of generative latent feature models which one can describe as generalized notions of Indian Buffet Processes(IBP). This work will not only provide a thorough description of the properties of existing models but also provide  a simple template to devise and analyze new models. That is to say, although one of the goals in this article is to promote a particular calculus, the results are presented in a fashion where it is not necessary for the reader interested primarily in implementation and new constructions  to have a command of that calculus.  

The IBP, and its generalizations, represent an exciting class of models well suited to handle high dimensional problems now common in this information age. These points are articulated in for instance \cite{Broderick1,Broderick3,Zoubin, Gorur,Griffiths1,GriffithsZ,Kang, KnowlesThesis, Miller, Palla, TehG, Thibaux, Williamson, Wood, Xu}, and for placement in a wider context see~\cite{orbanz1}. 
The IBP is based on the usage of conditionally independent Bernoulli random variables, one way to think of generalizations is to to use more general random variables. Of note in the current literature are models employing Poisson and Negative-Binomial random variables as described in for instance
\cite{Broderick2,Croy,Titsias,Zhou1,Zhou2}. These are generally coupled with gamma and beta process priors. 
Even in those specialized cases the properties of these models are not well understood, and the analysis conducted so far requires great care.  Additionally, for instance, it is unclear how replacing a gamma process with a generalized gamma process works in the model of \cite{Titsias}. Indeed a generalized gamma process is not conjugate in the Poisson model but it does have extra flexibility over its gamma process counterpart. This article describes general results  for IBP type processes based on variables having any distribution $G_{A},$  that has mass at zero and otherwise can be discrete or continuous. Examples of models employing continuous $A$ are given in~\cite{KnowlesThesis, KnowlesAAS}. Furthermore, the results are presented for general prior models based on completely random measures. That is to say the analysis does not require conjugacy.

The method proposed is  via the \emph{Poisson Partition Calculus} (PPC) devised by the author in~\cite{James2002,James2005}, and applied further for instance in~\cite{JamesNTR,JLP2}. 
The \textit{partition} calculus refers to a non-combinatorial approach to systematically derive expressions for combinatorial mechanisms derived from Bayesian exchangeable processes whereby $W_{1},\ldots, W_{n}|\mu$ are conditionally iid  and where $\mu$ is a random probability measure or more generally a random measure, that can be expressed as a functional of a Poisson random measure $N.$ One is interested in the posterior distribution of $\mu|W_{1},\ldots W
_{n},$ and the marginal distribution of $(W_{1},\ldots,W_{n}).$ 
Due to discreteness, the marginal distribution of $(W_{1},\ldots,W_{n}).$ 
can be viewed as an analogue of a Blackwell-MacQueen Polya urn scheme, which is the generative sampling scheme corresponding to the marginal joint probability measure induced by setting $\mu:=P$ to be a  Dirichlet process random probability measure, and modeling  $W_{1},\ldots, W_{n}|P$ to be iid $P.$  Furthermore, 
these structures describe random partitions of ${\{1,\ldots,n\}}$ by the number of unique values among $(W_{1},\ldots,W_{n})$ , which can be viewed as clusters, and the size of clusters based on the tied values. Loosely speaking these can be described as analogues of Chinese restaurant processes(CRP), most closely associated with the Dirichlet and Pitman-Yor processes which are used as priors/models in Bayesian non-parametric statistics and employed extensively in Statistical Machine Learning. See for instance \cite{IJ2003,Pit02,Pit06}. We shall provide a schematic for this method as it relates to latent feature models, but first describe more details for the IBP and its generalizations.   
\subsection{The Indian Buffet Process(IBP)}
The basic $\mathrm{IBP}(\theta)$ process, for $\theta>0,$
was ingeniously formulated by Griffiths and Ghahramani~\cite{GriffithsZ, Griffiths1} whereby a random binary matrix is formed with  $(\tilde{\omega}
_{k},k=1,\ldots)$ distinct 
features or attributes labeling an unbounded number of columns, and M rows, where each row $i$ represents the attributes/features/preferences possessed by a single individual by entering $1$ in the $(i,k)$ entry if the feature $\tilde{\omega}_{k}$ is possessed and $0$ otherwise. The matrix naturally indicates what features are shared among individuals. The  generative process to describe this is cast in terms of individual customers sequentially entering an Indian Buffet restaurant whereby the $i$-th customer chooses one of the $K$ already sampled dishes(indicates that customer shares features already exhibited by the previous customers) according to the most popular dishes already sampled, specifically with probability $m
_{j,i}/i,$ where $m_{j,i}$ denotes the number of customers who have already sampled dish $j$. Otherwise the customer chooses new dishes according to a $\mathrm{Poisson}(\theta/i)$ distribution. A key insight was made by Thibaux and Jordan~\cite{Thibaux}, which connected this generative process with a formal Bayesian non-parametric framework where $Z_{1},\ldots,Z_{M}|\mu$ are modelled as iid Bernoulli processes with base measure $\mu$ that is selected to have a Beta process prior distribution whose iid atoms $(\tilde{\omega}_{k})$ are obtained by specifying a proper probability $B_{0}$ as its base measure. The generative process is given by the distribution of $Z_{M+1}|Z_{1},\ldots, Z_{M}.$ In other words the basic IBP$(\theta)$ is generated from a random matrix with entries ${(b_{i,k})}$ where conditioned on the $(p_{j}),$ the points of Poisson random measure with mean intensity $\rho(s)=\theta s^{-1}\indic_{\{0<s<1\}},$ $b_{i,k}$ are independent Bernoulli $(p_{k})$ variables. It follows that one can represent each $Z_{i}=\sum_{k=1}^{\infty}b_{i,k}\delta_{\tilde{\omega}_{k}}$ and the Beta process $\mu=\sum_{k=1}^{\infty}p_{k}\delta_{\tilde{\omega}_{k}}.$

\begin{rem}
The concept of Beta processes was developed in the fundamental paper of Hjort~\cite{hjort} as it relates to Bayesian NP problems arising in survival analysis, Kim~\cite{KimY} is also an important reference in this regard. Naturally within this context is the related work of Doksum~\cite{Doksum}. However it is used in a rather different context in the IBP setting. Other uses of a beta process within a L\'evy moving process context can be found in James~\cite{James2005}
which apparently is a precursor to the kernel smoothed Beta process in \cite{Ren}. Spatial notions of Beta and other processes appear in~\cite{James2005, JamesNTR}.
\end{rem}
\subsection{Generalizations}
Using the theory of marked Poisson point processes we can easily describe the generalization of the IBP that includes the models already mentioned above. Simply replace the $((b_{i,k},p_{k}))$ with more general variables $((A_{i,k},\tau_{k}))$ where conditional on $(\tau_{k})$ $A_{i,k}$ are independent random variables(possibly vector valued) with distribution denoted as $G_{A}(da|\tau_{k}),$ where $\tau_{k}$ are the points of a Poisson random measure with more general L\'evy density $\rho(s|\omega),$ not restricted to $[0,1]$ and possibly depending on $\omega,$
 which reflects dependence on the distribution of the $(\tilde{\omega}_{k}),$ $B_{0},$ on a space $\Omega.$
Importantly, for each $i,$ 
$((A_{i,k},\tau_{k},(\tilde{\omega}_{k}))$
 are the points of a Poisson random measure with mean intensity
\begin{equation}
\label{markedmeasure}
G_{A}(da|s)\rho(s|\omega)dsB_{0}(d\omega).
\end{equation}
Furthermore we shall assume that $A_{i,k}$ can take on the value zero with positive probability. So relative to $G_{A}(\cdot|s),$ write $\pi_{A}(s)=\mathbb{P}(A\neq 0|s),$ and hence $1-\pi_{A}(s)=\mathbb{P}(A=0|s)$

The general construction is where now
\begin{equation}
Z_{i}=\sum_{k=1}^{\infty}A_{i,k}\delta_{{\tilde{\omega}}_{k}}{\mbox {and }}\mu=\sum_{k=1}^{\infty}\tau_{k}\delta_{\tilde{\omega}_{k}}.
\label{genpair}
\end{equation}
\begin{rem}
Notice that one can always represent $A_{i,k}=b_{i,k}A'_{i,k},$ where $b_{i,k}$ are Bernoulli variables and $A'_{i.k}$ is a general random variable that is not necessarily independent of $b_{i,k}.$ So that $Z_{i}$ in (\ref{genpair}) can be viewed as a weighted Bernoulli process. The dependent representation makes these more general than models of the form considered in \cite{KnowlesThesis, KnowlesAAS}.
\end{rem}
Here key to our exposition, it is important to note that $\mu$ can always be represented as 
\begin{equation}
\label{Nrep}
\mu(d\omega)=\int_{0}^{\infty}sN(ds,d\omega),
\end{equation}
where $N=\sum_{k=1}^{\infty}\delta_{\tau_{k},\tilde{\omega}_{k}}$ is a Poisson random measure with mean intensity 
$$
\mathbb{E}[N(ds,d\omega)]=\rho(s|\omega)dsB_{0}(d\omega):=\nu(ds,d\omega).
$$
As in James\cite{James2005}, $N$ takes its values in the space of boundedly finite measures, $\mathcal{M}.$
\begin{rem}
It is further noted that $\nu$ can be defined over any Polish space. This may be relevant in applications where $G_{A}$ is specified in terms of more general parameters, such as \cite{KnowlesAAS}, say $G_{A}(\cdot|s,\lambda).$  For instance $A|A\neq 0$ could correspond to a Normal distribution with mean and variance parameters $(\eta,\sigma).$ Such extensions from a technical point of view are easily handled and will not be discussed explicitly here in the univariate case. However the multivariate extension in section 5 covers this case.
\end{rem}
\begin{defin}
We say that $N$ is $\mathrm{PRM}(\rho B_{0}),$ or more generally $\mathrm{PRM}(\nu)$ and also write $\mathcal{P}(dN|\nu),$ when discussing calculations. $\mu$ is by construction a completely random measure and we shall specify its law by saying $\mu$ is 
$\mathrm{CRM}(\rho B_{0}),$ or $\mathrm{CRM}(\nu).$  Using this we shall say that $Z_{1},\ldots,Z_{M}|\mu$ are iid $\mathrm{IBP}(G_{A}|\mu),$ if they have the specifications in~(\ref{genpair}) and call the marginal distribution of $Z_{i}$ 
$\mathrm{IBP}(A,\rho B_{0})$ or equivalently $\mathrm{IBP}(A,\nu).$
\end{defin}
\begin{rem} Formally we choose $\rho$ to satisfy $\int_{0}^{\infty}\min(s,1)\rho(s|\omega)ds<\infty.$ However we allow both infinite activity models corresponding to $\int_{0}^{\infty}\rho(s|\omega)ds=\infty$ and finite activity/compound models where 
$\int_{0}^{\infty}\rho(s|\omega)ds=c(\omega)<\infty.$ In the latter case $\rho(s|\omega)/c(\omega)$ is a proper density. 
\end{rem}

\subsection{Outline}
Notice that in~(\ref{genpair}) the only difference in the setup is the choice of the distribution on $A_{i,k}.$ Obviously different choices of $G_{A}$ lead to different modeling capabilities and interpretations. We note further that the PPC is a method that does not rely on conjugacy. That is to say if it can be applied for one choice of $\rho,$ it can be applied for all choices of $\rho.$  In the forthcoming section we shall provide a schematic to derive the posterior distribution and related quantities for any $G_{A},$ and any $\rho.$ We will then state formally the posterior distribution for $N$ and $\mu$, marginal distribution of $Z_{i},$ and the distributions of $Z_{M+1}|Z_{1},\ldots, Z_{M}$
We shall also introduce classes of iid variables $(H_{i},X_{i})$ which can provide a representation for the marginal distribution of $Z
_{1},$ hence all $Z_{i}$ that can be implemented for most choices of $\rho.$ This is one of the key elements to describe the appropriate analogues of the Indian Buffet generative process that can be implemented broadly. We note that in order for this to apply, $G_{A}$ must at minimum admit a positive  mass at $0,$ or $\rho$ must be taken to be a finite measure.   Overall we shall show how to do the following for any $\rho,$ without taking limits, without guesswork and without combinatorial arguments, and lastly without long proofs. We shall then show details for all choices of $A_{i,k},$ Bernoulli, Poisson, Negative-Binomial,  that have been mentioned above.

\begin{itemize}
\item[$\bullet$]Generally apply the Poisson Calculus.
\item[$\bullet$]Describe $N|Z_{1},\ldots, Z_{M}$
\item[$\bullet$]Hence $\mu|Z_{1},\ldots,Z_{M}$
\item[$\bullet$]Describe the marginal structure $\mathbb{P}(Z_{1},\ldots,Z
_{M})$ via integration.
\item[$\bullet$]Provide descriptions for the marginal process as
$$
Z=\sum_{k=1}^{\xi(\varphi)}X_{k}\delta_{\tilde{\omega}_{k}},
$$
that via the introduction of variables $(H_{i},X_{i})$ leads to tractable sampling schemes for many $\rho$
\item[$\bullet$] Describe $Z_{M+1}|Z_{1},\ldots,Z_{M}$ 
\item[$\bullet$] These last two points lead to rather explicit descriptions of the marginal process that apply for many $\rho$ and hence lead to generative schemes that generalize the IBP feature selection scheme. 
\end{itemize}
In section 5 we will describe and provide parallel  details for a large class of multivariate models. Much of what has been described in the univariate case carries over quite clearly to the multivariate case. We shall also highlight the use of a beta-Dirichlet process described in~\cite{KimY2} as a natural extension of a Beta process in this setting.
\section{Poisson Latent Feature Calculus Schematic}One should first note that Poisson random measures are the building blocks for most discrete random processes used in the Bayesian Nonparametrics literature. Certainly this is true for any case where one employs completely random measures. 
\subsection{Basic updating operations using the PPC}
For our purposes here we will need two ingredients of the PPC that can be found as Propositions 2.1 and 2.2 in James~\citep[p. 6-8]{James2005}[see also \citep[p. 7-8]{James2002}]. These represent direct extensions of updating mechanisms for the Dirichlet and gamma processes employed in \cite{Lo1982, Lo1984, LoWeng89} and also~\cite{IJ2004,James2003}. 
Note again that 
\begin{equation}
N=\sum_{k=1}^{\infty}\delta_{\tau_{k},\tilde{\omega}_{k}}{\mbox { hence }}\mu(d\omega)=\int_{0}^{\infty}sN(ds,d\omega).
\label{relation}
\end{equation}
\subsubsection{The Laplace functional exponential change of measure}
First \cite[Proposition 2.1]{James2005} describes a Laplace functional exponential change of measure which describes updating $N$ through changing the mean measure $\nu,$ to $\nu_{f},$
\begin{equation}
\label{prop1}
{\mbox e}^{-N(f)}\mathcal{P}(dN|\nu)=\mathcal{P}(dN|\nu_{f}){\mbox e}^{-\Psi(f)},
\end{equation}
where in our setting $\nu_{f}(ds,d\omega)={\mbox e}^{-f(s,\omega)}\rho(s|\omega)dsB_{0}(d\omega)$ and ${\mbox e}^{-\Psi(f)}=\mathbb{E}_{\nu}[{\mbox e}^{-N(f)}],$ where
$$
\Psi(f)=\int_{\Omega}\int_{0}^{\infty}(1-e^{-f(s,\omega)})\rho(s|\omega)dsB_{0}(d\omega).
$$
This action changes $N$ from $\mathrm{PRM}(\nu)$ to $\mathrm{PRM}(\nu_{f}).$ 
\subsubsection{Disintegration involving the moment measure: decomposing $N$}
Next using $\nu_
{f}$ apply the moment measure calculation appearing in \cite[Proposition 2.2.]{James2005}
\begin{equation}
\label{prop2}
\left[\prod_{i=1}^{L}N(dW_{i})\right]\mathcal{P}(dN|\nu_{f})=\mathcal{P}(dN|\nu_{f},\mathbf{s,\omega})\prod_{\ell=1}^{K}{\mbox e}^{-f(s_{\ell},\omega_{\ell})}\nu(ds_{\ell},d\omega_{\ell})
\end{equation}
expressed over $K\leq L$ uniquely picked points $(s_{\ell},\omega_{\ell}),$ from $N.$ In other words conditioned on $N,$ the distribution of these points is determined by the joint measure 
$$
\prod_{\ell=1}^{K}N(ds_{\ell},d\omega_{\ell}).
$$

There are two important things to notice about~(\ref{prop2}). First is the meaning of $\mathcal{P}(dN|\nu_{f},\mathbf{s,\omega}),$ which corresponds to the law of the random measure
\begin{equation}
\label{Nsplit}
\tilde{N}+\sum_{\ell=1}^{K}\delta_{s_{\ell},\omega_{\ell}}
\end{equation}
where $\tilde{N}$ is also $\mathrm{PRM}(\nu_{f}).$ The quantity above is merely a decomposition of $N$ in terms of the unique points picked and those remaining . It says remarkably, and in fact is well known, that $\tilde{N}$
maintains the same law as $N.$ This translates into the following result, for any measurable function 
$g(\mathbf{W},N),$
$$
\int_{\mathcal{M}} g(\mathbf{W},N)\mathcal{P}(dN|\nu_{f},\mathbf{s,\omega})=
\int_{\mathcal{M}} g(\mathbf{W},N+\sum_{\ell=1}^{K}\delta_{s_{\ell},\omega_{\ell}})
\mathcal{P}(dN|\nu_{f}).
$$
Furthermore the marginal measure 
$$\prod_{\ell=1}^{K}\nu_{f}(ds_{\ell},d\omega_{\ell})=
\mathbb{E}_{\nu_{f}}[\prod_{i=1}^{L}N(dW_{i})]=\mathbb{E}_{\nu_{f}}[\prod_{\ell=1}^{K}N(ds_{\ell},d\omega_{\ell})],
$$ 
only depends on the $K$ unique values and not the multiplicities that might be associated with them. For a proof of this see~James \citep[Proposition 5.2]{James2005}.

\subsection{The joint distribution}
The form of the relevant likelihood can be deduced from the discussion above. However, it is perhaps instructive to first write it similar to the form used in \cite[Appendix A.1]{Zhou2}
$$
\prod_{i=1}^{M}\prod_{j=1}^{\infty}{[G_{A}(da_{i,j}|\tau_{j})]}^{\indic_{\{a_{i,j}\neq 0\}}}
{[1-\mathbb{P}(A\neq 0|\tau_{j})]}^{1-\indic_{\{a_{i,j}\neq 0\}}}
$$
which, setting $\pi_{A}(s)=\mathbb{P}(A\neq 0|s)$ can be written as 
$$
{\mbox e}^{-\sum_{j=1}^{\infty}[-M\log(1-\pi_{A}(\tau_{j}))]}\prod_{i=1}^{M}\prod_{j=1}^{\infty}{\left[\frac{G_{A}(da_{i,j}|\tau_{j})}{
1-\pi_{A}(\tau_{j})}\right]}^{\indic_{\{a_{i,j}\neq 0\}}}.
$$ 
It then follows by, arguments similar to \cite[Appendix A.1]{Zhou2}, and augmenting $\mu,$ that one can write the joint distribution of  $((Z_{1},\ldots, Z_{M}), (J_{1},\ldots,J_{K}),N),$ where $(J_{1}=s_{1},\ldots, J_{K}=s_{K})$ are the unique jumps, as,
\begin{equation}
\label{Alikelihood}
{\mbox e}^{-N(f_{M})}\mathcal{P}(dN|\nu)\prod_{\ell=1}^{K}N(ds_{\ell},d\omega_{\ell})\prod_{i=1}^{M}{\left[\frac{G_{A}(da_{i,\ell}|s_{\ell})}{
1-\pi_{A}(s_{\ell})}\right]}^{\indic_{\{a_{i,\ell}\neq 0\}}}.
\end{equation}

where $f_{M}(s,\omega)=-M\log(1-\pi_{A}(s)].$ 

Now apply the exponential change of measure in (\ref{prop1}) to obtain 
$$
\mathcal{P}(dN|\nu_{f_{M}}){\mbox e}^{-\Psi(f_{M})}\prod_{\ell=1}^{K}N(ds_{\ell},d\omega_{\ell})\prod_{i=1}^{M}{\left[\frac{G_{A}(da_{i,\ell}|s_{\ell})}{
1-\pi_{A}(s_{\ell})}\right]}^{\indic_{\{a_{i,\ell}\neq 0\}}}.
$$
Noticing that,
$$
{\mbox e}^{-f_{M}(s_{\ell},\omega_{\ell})}={[1-\pi_{A}(s_{\ell})]}^{M},
$$ 
this operation results in the law of $N$ changing from $\mathrm{PRM}(\nu)$ to $\mathrm{PRM}(\nu_{f_{M}}),$ that is  $\nu(ds,d\omega)$ is changed to 
$$\nu_{f_{M}}(ds,d\omega)={\mbox e}^{-f_{M}(s,\omega)}\nu(ds,d\omega)=
{[1-\pi_{A}(s)]}^{M}\rho(s|\omega)B_{0}(d\omega)ds.
$$

Now apply the disintegration in~(\ref{prop2}) to obtain the desired final form of the joint distribution
\begin{equation}
\label{mainjoint}
\mathcal{P}(dN|\nu_{f_{M}},\mathbf{s,\omega}){\mbox e}^{-\Psi(f_{M})}
\prod_{\ell=1}^{K}{\mbox e}^{-f_{M}(s_{\ell},\omega_{\ell})}\rho(s_{\ell}|\omega_{\ell})B_{0}(d\omega_{\ell}))
\prod_{i=1}^{M}h_{i,\ell}(s_{\ell})
\end{equation}
or
\begin{equation} 
\label{mainjoint3}
\mathcal{P}(dN|\nu_{f_{M}},\mathbf{s,\omega}){\mbox e}^{-\Psi(f_{M})}
\prod_{\ell=1}^{K}
{[1-\pi_{A}(s_{\ell})]}^{M}
\rho(s_{\ell}|\omega_{\ell})B_{0}(d\omega_{\ell}))
\prod_{i=1}^{M}h_{i,\ell}(s_{\ell})
\end{equation}
where
\begin{equation}
\label{hfunction}
h_{i,\ell}(s)={\left[\frac{G_{A}(da_{i,\ell}|s)}{
1-\pi_{A}(s)}\right]}^{\indic_{\{a_{i,\ell}\neq 0\}}},
\end{equation}
and
\begin{equation}
\Psi(f_{M})=\int_{\Omega}\int_{0}^{\infty}(1-{[1-\pi_{A}(s)]}^{M})\rho(s|\omega)dsB_{0}(d\omega).
\label{psifunctionM}
\end{equation}

\begin{rem}
Note in order to emphasize a notion of partial conjugacy, for any L\'evy density $\rho$ and $\beta\ge 0$ we could without loss of generality define a Levy density
$$
\tilde{\rho}_{\pi_{A},\beta}(s|\omega)={[1-\pi_{A}(s)]}^{\beta}\rho(s|\omega)
$$
which we would assign to the prior measure $\mu$
Note of course that $\tilde{\rho}_{\pi_{A},0}(s|\omega)=\rho.$
Hence it follows that for $\beta\ge 0,$ and $\nu(ds,d\omega)=\tilde{\rho}_{\pi_{A},\beta}(s|\omega)B_{0}(d\omega)ds$ that 
$$
\nu_{f_{M}}(ds,d\omega)/(B_{0}(d\omega)ds)=\tilde{\rho}_{\pi_{A},\beta+M}(s|\omega):={[1-\pi_{A}(s)]}^{\beta+M}\rho(s|\omega)
$$
We will use some variant of this for illustration in the special cases of Poisson, Binomial and Negative-Binomial models. 
\end{rem}

\begin{rem}It is to be emphasized that once a quantity such as~(\ref{mainjoint}) has been calculated this contains all the ingredients for a formal posterior analysis. What remains are arguments via Bayes rule and finite-dimensional refinements involving particular choices of $G_{A},$ $\rho$ etc.

\end{rem}

\subsection{Describing posterior and marginal quantities}
Now from (\ref{Nsplit}) and (\ref{mainjoint}) one can immediately deduce that for fixed $J_{\ell}=s_{\ell},$ that $\mathcal{P}(dN|\nu_{f_{M}},\mathbf{s,\omega})$ corresponds to the case where $\tilde{N}:=N_{M}$ is $\mathrm{PRM}(\nu_{f_{M}})$
hence $\mu$ can be expressed as $\mu_{M}+\sum_{\ell=1}^{K}s_{\ell}\delta_{\omega_{\ell}},$ where $\mu_{M}$ is a $\mathrm{CRM}(\nu_{f_{M}}).$ 
What remains to complete the posterior distribution of $N|Z_{1},\ldots,Z_{M}$ is to find the distribution of $(J_{\ell})$ given $Z_{1},\ldots,Z_{M},$ Integrating out $N$ in~(\ref{mainjoint}) leads to the joint distribution of $((Z_{i}),(J_\ell)),$
\begin{equation}
\label{mainjoint2}
{\mbox e}^{-\Psi(f_{M})}
\prod_{\ell=1}^{K}{[1-\pi_{A}(s_{\ell})]}^{M}\rho(s_{\ell}|\omega_{\ell})B_{0}(d\omega_{\ell}))
\prod_{i=1}^{M}h_{i,\ell}(s_{\ell}).
\end{equation}
Note this is now in the usual finite dimensional Bayesian setup. 
Hence $(J_\ell)|(Z_{i})$ are conditionally independent with density, 
\begin{equation}
\label{marginalJ} 
\mathbb{P}_{\ell,M}(J_{\ell}\in ds)\propto {[1-\pi_{A}(s)]}^{M}\rho(s|\omega_{\ell})
\prod_{i=1}^{M}h_{i,\ell}(s)ds
\end{equation}
and the marginal of $Z_{1},\ldots,Z_{M},$ is 
\begin{equation}
\label{margZ}
\mathbb{P}(Z_{1},\ldots,Z_{M})={\mbox e}^{-\Psi(f_{M})}\prod_{\ell=1}^{K}\left[\int_{0}^{\infty}
{[1-\pi_{A}(s)]}^{M}\rho(s|\omega_{\ell})
\prod_{i=1}^{M}h_{i,\ell}(s)ds\right]B_{0}(d\omega_{\ell}).
\end{equation}
\subsection{The general posterior distribution}
We now summarize our results. Throughout we use, 
$$\nu_{f_{M}}(ds,d\omega):=\rho_{M}(s|\omega)B_{0}(d\omega)ds.
$$
where $\rho_{M}(s|\omega)={[1-\pi_{A}(s)]}^{M}\rho(s|\omega).$ Note that $\rho_{M}$ depends on $A$ and so will differ over various models. Since it should be clear from context, we shall suppress this dependence within the notation. 
\begin{thm}
\label{genthm}
Suppose that $Z_{1},\ldots,Z_{M}|\mu$ are iid $\mathrm{IBP}(G_{A}|\mu),$ $\mu$ is $\mathrm{CRM}(\rho B_{0}).$ Then
\begin{enumerate}
\item[(i)]The posterior distribution of $N|Z_{1},\ldots,Z_{M}$ is equivalent to the distribution of
$$
N_{M}+\sum
_{\ell=1}^{K}\delta_{J_{\ell},\omega_{\ell}}
$$
where $N_{M}$ is $\mathrm{PRM}(\rho_{M} B_{0})$ and the distribution of the conditionally independent $(J_{\ell})$ is given by~(\ref{marginalJ}).
\item[(ii)]The posterior distribution of $\mu|Z_{1},\ldots,Z_{M}$ is equivalent to the distribution of
\begin{equation}
\label{genMu}
\mu_{M}+\sum
_{\ell=1}^{K}J_{\ell}\delta_{\omega_{\ell}}
\end{equation}
where $\mu_{M}$ is $\mathrm{CRM}(\rho_{M} B_{0})$
\item[(iii)] The marginal distribution of $(Z_{1},\ldots,Z_{M})$ is given by~(\ref{margZ}).
\end{enumerate}
\end{thm}
This immediately leads to the next result
\begin{prop}
\label{Zproposition}
Suppose that $Z_{1},\ldots,Z_{M},Z
_{M+1}|\mu$ are iid $\mathrm{IBP}(G_{A}|\mu),$
,
where $\mu$ is $\mathrm{CRM}(\rho B_{0}).$ Then from results in Theorem~\ref{genthm}, equation (\ref{genMu}) shows that the distribution of 
$Z_{M+1}|Z_{1},\ldots,Z_{M}$ is equivalent to the distribution of 
$$
\tilde{Z}_{M+1}+\sum_{\ell=1}^{K}A_{\ell}\delta_{\omega_{\ell}}
$$
where $\tilde{Z}_{M+1}$ is  $\mathrm{IBP}(A,\rho_{M} B_{0}),$ and each $A_{\ell}|J_{\ell}=s$ has distribution $G_{A}(da|s),$ and the marginal distribution of $J_{\ell}$ is specified by  (\ref{marginalJ}). That is to say $A_{\ell}$ has the distribution $\int_{0}^{\infty}G_{A}(da|y)\mathbb{P}_{\ell,M}(dy).$ $(\omega_{\ell})$ are the already chosen features.   
\end{prop}
\begin{rem} It should be evident that it easy to make adjustments for the case where the $Z_{1},\ldots,Z_{M}|\mu$ are independent rather than iid. Furthermore semi-parametric models may be treated as in~\cite{James2003,IJ2004}, see also~\cite{James2005}.
\end{rem}

\subsection{Details for the Bernoulli, Poisson, and Negative-Binomial IBP type processes}
Before we continue with a general discussion, we take a look at specifics in the Poisson, Bernoulli, and Negative-Binomial cases. Note while we can just employ Theorem~\ref{genthm} and Proposition~\ref{Zproposition} directly and list out results, we provide some specific developments of those general arguments used in these special cases for illustrative purposes. 
Note if $G_{A}(\cdot|s)$ corresponds to  Poisson$(bs):=\mathrm{Poi}(bs),$  Bernoulli$(s):=\mathrm{Ber}(s)$, or Negative-Binomial$(r,s)$, denoted as $\mathrm{NB}(r,s)$, random variable then respectively the probability mass function with arguments ${a_{i,\ell}}$
is in the Poisson$(bs)$ case
$$
\mathbb{P}(A_{i,\ell}=a_{i,\ell}|s)= \frac{b^{a_{i,\ell}}s^{a_{i,\ell}}}{a_{i,\ell}!}{\mbox e}^{-bs}, a_{i,\ell}=0,1,\ldots
$$
In the Bernoulli$(s)$ case,
$$
\mathbb{P}(A_{i,\ell}=a_{i,\ell}|s)=s^{a_{i,\ell}}{(1-s)}^{a_{i,\ell}}, a_{i,\ell}=0,1,
$$
and in the Negative-Binomial $(r,s)$ case 
$$
\mathbb{P}(A_{i,\ell}=a_{i,\ell}|s)={a_{i,\ell}+r-1\choose a
_{i,\ell}}s^{a_{i,\ell}}{(1-s)}^{r}, a_{i,\ell}=0,1,\ldots
$$
Hence in the respective cases $\mathbb{P}(A=0|s)=1-\pi_{A}(s)$ takes the values 
$$
{\mbox e}^{-bs},1-s{\mbox { and }} {(1-s)}^{r},
$$
where in the latter two cases $s\in[0,1].$ Furthermore, $h_{i,\ell}(s)$ given in~(\ref{hfunction}) takes the respective forms
$$
\frac{b^{a_{i,\ell}}s^{a_{i,\ell}}}{a_{i,\ell}!}, {\left(\frac{s}{1-s}\right)}^{a_{i,\ell}},{\mbox { and }}{a_{i,\ell}+r-1\choose a
_{i,\ell}}s^{a_{i,\ell}}
$$
Now setting $c_{\ell,M}=\sum_{i=1}^{M}a_{i,\ell},$ it follows that $\prod_{i=1}^{M}h_{i\ell}(s),$ takes the form
$$
\frac{b^{c_{\ell,M}}s^{c_{\ell,M}}}{\prod_{i=1}^{M}a_{i,\ell}!}, 
{\left(\frac{s}{1-s}\right)}^{c_{\ell,M}},{\mbox { and }}s^{c_{\ell,M}}\prod_{i=1}^{M}{a_{i,\ell}+r-1\choose a
_{i,\ell}}
$$
Note to simplify notation we shall do the remaining calculations for homogeneous $\rho.$ Readers can easily make the adjustments to the case of $\rho(s|\omega).$ First some notation and known facts are introduced. 
\begin{rem}
It is important to note that the points and notation described next can be found in Gnedin and Pitman~\cite{GnedinPitman2} and James~\cite{JamesNTR}. Which alludes to connection between the Poisson, Bernoulli, and Negative-Binomial IBP models and Regenerative Compositions/ Neutral to the Right processes. Many explicit calculations and examples for these IBP models can be deduced from those works.  We encourage the reader to take a look at those works for such details. 
\end{rem}
Let $T=\sum_{k=1}^{\infty}\Delta_{k}$ be an infinitely divisible random variable where $(\Delta_{k})$ are the ranked jumps of a subordinator determined by the Levy density $\tau_{\infty},$ then it follows that $\mathbb{E}[{\mbox e}^{-\lambda T}]={\mbox e}^{-\phi(\lambda)},$ where
$$
\phi(\lambda)=\phi_{\infty}(\lambda):=\int_{0}^{\infty}(1-{\mbox e}^{-\lambda s})\tau_{\infty}(s)ds.
$$
From $(\Delta_{k})$ one can construct an infinitely divisible random variable $U=\sum_{k=1}^{\infty}(1-{\mbox e}^{-\Delta_{k}})$ whose points are determined by a Levy density 
$$\tau_{[0,1]}(u)=(1-u)^{-1}\tau_{\infty}(-\log(1-u)), u\in [0,1]$$
and there is also the converse relation
$$\tau_{\infty}(y)={\mbox e}^{-y}\tau_{[0,1]}(1-{\mbox e}^{-y}), y\in (0,\infty).$$ Furthermore
$$ \phi(\lambda)=\phi_{[0,1]|}(\lambda)=\int_{0}^{1}(1-{(1-u)}^{\lambda})\tau_{[0,1]}(du)$$
which is equivalent to 
$$
\phi_{[0,1]}(\lambda)=\int_{0}^{1}\lambda{(1-u)}^{\lambda-1}
\left[\int_{u}^{1}\tau_{[0,1]}(dv)\right]du.$$
\begin{rem}
Note taking $\tau_{[0,1]}(u)=\theta u^{-1}(1-u)^{\beta-1}$ gives the homogeneous beta process, and it can be read from~Gnedin\cite{GnedinS} that for an integer $M,$
$$
\phi(M):=\phi_{[0,1]}(M)=\sum_{k=1}^{M}\frac{\theta}{\beta+k-1},
$$
which is the term appearing in $\mathbb{P}(Z_{1},\ldots, Z_{M})$ in the Bernoulli IBP case under a beta process CRM. 
\end{rem}

\subsubsection{Poisson IBP Case}
Call $Z_{1},\ldots,Z_{M}|\mu$ in this model, iid $(\mathrm{Poi}(bs)|\mu)$ processes. It follows that if $\mu$ is determined by the Levy density $\rho(s)$ then it follows for each M that,
$$
\rho_{M}(s)={\mbox e}^{-bM s}\rho(s):=\tilde{\rho}_{bM}(s)
$$
for any L\'evy density $\rho.$ 
Hence denote the marginal distribution of $Z_{1}$ as $\mathrm{IBP}(\mathrm{Poi}(bs),\rho B_{0}).$
Furthermore, 
$$
\psi(f_{M}):={\phi}_{\infty}(bM):=\int_{0}^{\infty}(1-{\mbox e}^{-bMs})\rho(s)ds,
$$
where $\tau_{\infty}:=\rho$.
Hence setting $C_{\ell,b}={b^{c_{\ell,M}}}/{\prod_{i=1}^{M}a_{i,\ell}!}$ the joint distribution in~(\ref{mainjoint3})
is
\begin{equation}
\label{PoissonJoint}
\mathcal{P}(dN|\tilde{\rho}_{bM}B_{0},\mathbf{s,\omega}){\mbox e}^{-{\phi}_{\infty}(bM)}
\prod_{\ell=1}^{K}C_{\ell,b}
s^{c_{\ell,M}}_{\ell}{\mbox e}^{-(bM)s_{\ell}}
\rho(s_{\ell})B_{0}(d\omega_{\ell}).
\end{equation}

It should be  evident from (\ref{PoissonJoint}) that the marginal distribution in the Poisson case is given as 
\begin{equation}
\label{poissonZ}
\mathbb{P}(Z_{1},\ldots,Z_{M})={\mbox e}^{-{\phi}_{\infty}(bM)}
\prod_{\ell=1}^{K}C_{\ell,b}\left[\kappa_{c_{\ell,M}}(\tilde{\rho}_{bM})\right]B_{0}(d\omega_{\ell}),
\end{equation}
where
$$
\kappa_{j}(\tilde{\rho}_{bM})=\int_{0}^{\infty}
y^{j}{\mbox e}^{-(bM)y}
\rho(y)dy.
$$
In addition, the posterior distribution of $N|Z_{1},\ldots, Z_{M}$ corresponds to that of the random measures
\begin{equation}
\label{Poissondecomp}
N_{M}+\sum_{\ell=1}^{K}\delta_{J_{\ell},\omega_{\ell}}
\end{equation}
where $N_{M}$ is $\mathrm{PRM}(\tilde{\rho}_{bM}B_{0})$ and independent of this, the distribution of $(J_{\ell})|Z_{1},\ldots,Z_{M}$ are conditionally independent with density
\begin{equation}
\label{PoissonJumps}
\mathbb{P}(J_{\ell}\in ds)\propto s^{c_{\ell,M}}{\mbox e}^{-(bM)s}
\rho(s).
\end{equation}
This implies that $\mu|Z_{1},\ldots, Z_{M}$ is equivalent in distribution to $\mu_{M}+\sum_{\ell=1}^{K}J_{\ell}\delta_{\omega_{\ell}}$ where $\mu_{M}$ is $\mathrm{CRM}(\tilde{\rho}_{bM}B_{0}).$ Note further if we refer to the marginal distribution of $Z_{1}$ as $\mathrm{IBP}(\mathrm{Poi}(bs),{\rho}B_{0})$ our analysis shows that $Z_{M+1}|Z_{1},\ldots, Z_{M}$ can be written as 
$$
\tilde{Z}_{M+1}+\sum_{\ell=1}^{K}A_{\ell}\delta_{\omega_{\ell}},
$$
where $\tilde{Z}_{M+1}$ is $\mathrm{IBP}(\mathrm{Poi}(bs),\tilde{\rho}_{bM}B_{0}),$ and $A_{\ell}|(J_{\ell})$ are conditionally independent Poisson$(bJ_{\ell}),$ variables. 

\begin{rem} It is easy to see that although we cannot expect $N|Z_{1},\ldots,Z_{M}$ and hence $\mu|Z_{1},\ldots,Z_{M}$ to be fully conjugate models, there is a sense of partial conjugacy in the following manner. Suppose that the mean intensity of 
$N$ is $\mathbb{E}[N(ds,d\omega)]={\mbox e}^{-\beta s}\rho(s)B_{0}(d\omega),$ then within the posterior decomposition~(\ref{Poissondecomp}) the PRM $N_{M}$ has mean intensity  
$$\mathbb{E}[N_{M}(ds,d\omega)]={\mbox e}^{-(\beta+bM) s}\rho(s)B_{0}(d\omega).$$
\end{rem}

\subsubsection{Gamma, Stable and generalized gamma process priors.}
We now give details for the Gamma-Poisson model that was investigated in Titsias~\cite{Titsias} and then move beyond that to the case where the gamma process is replaced by a generalized gamma process.  There are no known results for this flexible and natural extension, We should note here that the gamma process model of Titsias bears some similarities to Lo (1982)~\cite{Lo1982}. The latter is a reference which is apparently unknown within the general IBP literature. Lo~\cite{Lo1982} shows that if a weighted gamma process is used as a prior for the mean intensity of a Poisson process then its posterior distribution is also a weighted gamma process given the observations from $n$ Poisson Processes. In this sense it can be seen as a precursor to the class of IBP models.

We describe a scalar version of a weighted gamma process, which is just based on a gamma random variable $T$ with shape parameter $\theta$ and scale $1/\beta,$ with law denoted as Gamma$(\theta,1/\beta).$ That is to say if $T$ is such a variable then $\beta T$ is Gamma$(\theta,1).$ This corresponds to $N$ a PRM with mean intensity
$$\mathbb{E}[N(ds,d\omega)]=\tilde{\rho}_{\theta,\beta}=\theta s^{-1}{\mbox e}^{-\beta s}B_{0}(d\omega),$$ and hence $\mu$ is a weighted gamma process with parameters $(\beta,\theta B_{0})$ Note in Lo's case one takes $\beta$ to be a function $\beta(\omega),$ which again presents no extra difficulties here. From the above discussion it follows that 
$\mathbb{E}[N_{M}(ds,d\omega)]=\theta s^{-1}{\mbox e}^{-(\beta+bM) s}B_{0}(d\omega),$ and hence $\mu_{M}$ is a weighted gamma process with parameters $(\beta+bM,\theta B_{0}).$ Furthermore the distribution of $(J_{\ell}),$ is such that $G_{\ell}:=(\beta+bM)J_{\ell}$ are conditionally independent Gamma$(c_{\ell,M},1)$ random variables. Hence the $A_{\ell}$ are negative binomial variables as described in Titsias.  Furthermore the L\'evy exponent is $\phi_{\infty}(bM):=\theta\log(1+bM/\beta)$ and  hence the joint distribution $\mathbb{P}(Z_{1},\ldots, Z_{M})$ is given by 
$$
{\left(\frac{1}{\beta+bM}\right)}^{\theta+c_{M}}\beta^{\theta}\theta^{K}\prod_{\ell=1}^{K}\Gamma(c_{\ell,M})C_{\ell,b}B_{0}(d\omega_{\ell})
$$
where $c_{M}=\sum_{\ell=1}^{K}c_{\ell,M}.$ Note under $Z_{M+1}|Z_{1},\ldots, Z_{M};$ 
$$\tilde{Z}_{M+1}\sim\mathrm{IBP}(\mathrm{Poi}(bs),\tilde{\rho}_{\theta,\beta+bM}B_{0})$$
and the $A_{\ell}|G_{\ell}$ are Poisson$(bG_{\ell}/(bM+\beta)).$ Hence the marginals of $A_{\ell}$ are Negative-Binomial.

Now suppose that for $0<\alpha<1$ $N$ is PRM with mean intensity 
$$
\mathbb{E}[N(ds,d\omega)]/(B_{0}(d\omega)ds)=\tau_{\infty}(s):=\rho_{\alpha,\beta}(s):=\frac{\alpha s^{-\alpha-1}{\mbox e}^{-s\beta}}{\Gamma(1-\alpha)}
$$
Then $\mu$ is a generalized gamma process  with parameters $(\alpha,\beta, B_{0})$ corresponding to a random variable $T$ with density ${\mbox e}^{-[t\beta-\beta^{\alpha}]}f_{\alpha}(t),$  where $f_{\alpha}$ is the density of a positive Stable random variable with index $0<\alpha<1.$ Furthermore 
$$\phi(bM):=\psi_{\alpha,\beta}(bM)=[(\beta+bM)^{\alpha}-\beta^{\alpha}].$$

It follows that $\mu_{M}$ is a generalized gamma process with parameters $(\alpha,\beta+bM, B_{0})$ and the jumps satisfy $G_{\ell}:=(bM+\beta)J_{\ell}$ are independent Gamma  $(c_{\ell,M}-\alpha,1)$ random variables. Hence  $\mathbb{P}
(Z_{1},\ldots, Z_{M})$ can be expressed via 
$$
{\mbox e}^{-\psi_{\alpha,\beta}(bM)}\alpha^{K}
{(bM+\beta)}^{K\alpha-c_{M}}\prod_{\ell=1}^{K}\frac{\Gamma(c_{\ell.M}-\alpha)}
{\Gamma(1-\alpha)}C_{\ell,b}.
$$
Note under $Z_{M+1}|Z_{1},\ldots, Z_{M};$ $\tilde{Z}_{M+1}$ is $\mathrm{IBP}(\mathrm{Poi}(bs),\rho_{\alpha,\beta+bM}B_{0})$
and the $A_{\ell}|G_{\ell}$ are Poisson$(bG_{\ell}/(bM+\beta)).$ Hence the marginals of $A_{\ell}$ are again Negative-Binomial. 
\begin{rem}Note it is interesting to compare the generalized gamma case in this setting to cases of L\'evy moving average processes~\citep[section 4.4.1,p.23]{James2005}, neutral to the right processes~\citep[p. 433,section 6.1]{JamesNTR}  and normalized random measures from an application of~\cite{JLP2} as it appears in \citep[example 2.4, p.341]{FavaroTeh}.
\end{rem}
\subsubsection{Bernoulli case-the oiginal IBP}
When $A|s$ is $Ber(s)$ this corresponds to the original IBP$(\mu)$ structure where details have been worked out primarily in the case where $\mu$ is some variant of a beta process,~see \cite{Thibaux,TehG}. In those cases results are deduced from \cite{KimY}. Here we provide details for any $\rho$. Again  in order to indicate how updates are made we write
for any $\beta\ge 0,$ and any $\rho$ defined on $[0,1],$ set 
$$
\hat{\rho}_{\beta}(s)={[1-s]}^{\beta}\rho(s).
$$
The results are stated using this choice, however there is no loss of generality to set $\beta=0.$
Assuming that $\mathbb{E}[N(ds,d\omega)]=\hat{\rho}_{\beta}(s)B_{0}(d\omega)ds,$ the joint distribution of $(N,(J_{\ell}, Z_{1},\ldots, Z_{M})$
can be expressed as
$$
\mathcal{P}(dN|\hat{\rho}_{\beta+M}B_{0},(s,\omega)){\mbox e}^{-\hat{\phi}_{[0,1]}(\beta+M)}\prod_{\ell=1}^{K}{\left(\frac{s_{\ell}}{1-s_{\ell}}\right)}^{c
_{\ell,M}}\hat{\rho}_{\beta+M}(ds_{\ell}|\omega_{\ell})B_{0}(d\omega_{\ell})
$$
where with respect to $\rho$
\begin{eqnarray}
\hat{\phi}_{[0,1]}(\beta+M)&=&\phi_{[0,1]}(\beta+M)-\phi_{[0,1]}(\beta)\\\nonumber
&=& \int_{0}^{1}(1-{(1-s)}^{M})\hat{\rho}_{\beta}(s)ds
\end{eqnarray}

Integrating out $N$ leaves a joint distribution of the latent jumps, and the latent feature structure, which we write in a different way as, 
$$
{\mbox e}^{-\hat{\phi}_{[0,1]}(\beta+M)}\prod_{\ell=1}^{K}{s_{\ell}}^{c
_{\ell,M}}{(1-s_{\ell})}^{\beta+M-c_{\ell,M}}\rho(ds_{\ell}|\omega_{\ell})B_{0}(d\omega_{\ell}),
$$
which can be expressed as
$$
\left[\prod_{\ell=1}^{K}\mathbb{P}(J_{\ell}\in d_{s_{\ell}}|c_{\ell,M})\right]\mathbb{P}(Z_{1},\ldots,Z_{M})
$$
where for $\hat{\kappa}_{c_{\ell,M}}(\rho)=\int_{0}^{1}{s}^{c
_{\ell,M}}{(1-s)}^{-c_{\ell,M}}\rho(s)ds,$
$$
\mathbb{P}(Z_{1},\ldots,Z_{M})={\mbox e}^{-\hat{\phi}_{[0,1]}(\beta+M)}\prod_{\ell=1}^{K}
\hat{\kappa}_{c_{\ell,M}}(\hat{\rho}_{\beta+M})B_{0}(d\omega_{\ell})
$$
and hence showing that the distribution of $J_{\ell}$ has density proportional to $${[s/(1-s)]}^{c_{\ell,M}}\rho_{\beta+M}(s).$$
It follows that $N|Z_{1},\ldots, Z_{M}$ can be expressed as 
$N_{M}+\sum_{\ell=1}^{K}\delta_{J_{\ell},\omega_{\ell}},$ 
where $N_{M}$ is a PRM with $\mathbb{E}[N_{M}(ds,d\omega)]=\rho_{\beta+M}(s)dsB_{0}(d\omega)$ and the distribution of the jumps $(J_{\ell})$ have been described above. It follows that $Z_{1}$ under these specifications is marginally $\mathrm{IBP}(\mathrm{Ber}, \hat{\rho}_{\beta}B_{0})$ which is equivalent in distribution to $\sum_{k=1}^{\xi(\kappa(\hat{\rho}_{\beta}))}\delta_{\tilde{\omega}_{k}},$ where 
$\xi(\kappa(\hat{\rho}_{\beta}))$ is a Poisson random variable with mean 
$
\kappa(\hat{\rho}_{\beta})=\int_{0}^{1}s\hat{\rho}_{\beta}(s)ds.
$
$Z_{M+1}|Z_{1},\ldots, Z_{M}$ is equivalent to $\tilde{Z}_{M+1}+\sum_{\ell=1}^{K}A_{\ell}\delta_{\omega_{\ell}},$ where $\tilde{Z}_{M+1}$ is $\mathrm{IBP}(\mathrm{Ber}, \hat{\rho}_{\beta+M}B_{0})$
and the $A_{\ell}$ are independent Bernoulli variables with success probability 
$$
\mathbb{E}[J_{\ell}|c_{\ell,M}]=\frac{\int_{0}^{1}{s}^{c
_{\ell,M}+1}{(1-s)}^{-c_{\ell,M}}\hat{\rho}_{\beta+M}(s)ds}{\hat{\kappa}_{c_{\ell,M}}(\hat{\rho}_{\beta+M})}.
$$
In particular if $\rho_{\beta}(s)=\theta s^{-\alpha-1}(1-s)^{\beta+\alpha-1},$ as in \cite{TehG}, the $(J_{\ell})$ are independent $\mathrm{Beta}(
c_{\ell,M}-\alpha,\beta+\alpha+M-c_{\ell,M})$ random variables then the customer $M+1$ chooses an existing dish $\omega_{\ell}$ with probability
$$
\mathbb{E}[J_{\ell}|c_{\ell,M}]=\frac{c_{\ell,M}-\alpha}{M+\beta}.
$$
Set $\beta=1$ and $\alpha=0$ to recover the original IBP model of Griffiths and Ghahramani~\cite{GriffithsZ}. Note for this case,  one can choose $\beta>-\alpha$, such that $\beta+\alpha>0,$ which are the specifications for the case where $\mu$ is a stable-Beta process discussed in~\cite{TehG}. See section \ref{bD} for a multivariate extension of this process we call a stable-Beta-Dirichlet process. 
\begin{rem}One can compare the beta process in this setting with that of a smoothed spatial beta process arising in the context of L\'evy moving averages \cite[section 4.4.2, p. 25-26]{James2005}. 
\end{rem}
\subsubsection{Negative-Binomial}
As mentioned previously several authors have investigated the case where  $A|s$ is $NB(r,s),$ and primarily this has been coupled with a CRM that is a beta process.  With some efforts~\cite{Broderick2} were able to show that the beta process was conjugate in this setting,[see also \cite{Broderick4}]. Again here it is demonstrated that applying the PPC makes analysis of considerably generalized notions of this case rather transparent.  As should be expected, the NB case bears some similarities to the Bernoulli case, Here we start with $N$ based on a Levy density $\hat{\rho}_{\beta r}(s)=[1-s]^{\beta r}\rho(s)$
The joint distribution can be expressed as
\begin{equation}
\label{NBjoint}
\mathcal{P}(dN|\hat{\rho}_{(\beta+M)r}B_{0},(s,\omega))
\left[\prod_{\ell=1}^{K}\mathbb{P}(J_{\ell}\in d_{s_{\ell}}|c_{\ell,M})\right]\mathbb{P}(Z_{1},\ldots,Z_{M})
\end{equation}
where $
\left[\prod_{\ell=1}^{K}\mathbb{P}(J_{\ell}\in d_{s_{\ell}}|c_{\ell,M})\right]\mathbb{P}(Z_{1},\ldots,Z_{M})
$
is equivalent to 
$$
{\mbox e}^{-\hat{\phi}_{[0,1]}((\beta+M)r)}
\prod_{\ell=1}^{K}{s_{\ell}}^{c
_{\ell,M}}{(1-s_{\ell})}^{(M+\beta)r}\rho(ds_{\ell})R_{\ell,M}B_{0}(d\omega_{\ell})
$$
where now 
$$
R_{\ell,M}=\prod_{i=1}^{M}{a_{i,\ell}+r-1\choose a
_{i,\ell}}.
$$
It is evident that the distribution of the $J_{\ell}$ has density proportional to ${s}^{c
_{\ell,M}}\hat{\rho}_{(\beta+M)r}(s).$ Furthermore setting $\kappa_{j}(\rho)=\int_{0}^{1}s^{j}\rho(s)ds,$ it follows that 
$$
\mathbb{P}(Z_{1},\ldots,Z_{M})=
{\mbox e}^{-\hat{\phi}_{[0,1]}((\beta+M)r)}
\prod_{\ell=1}^{K}\kappa_{c_{\ell,M}}(\hat{\rho}_{(\beta+M)r)})
R_{\ell,M}B_{0}(d\omega_{\ell}).
$$
We quickly identify other relevant quantities. $N_{M}$ is a PRM with mean satisfying $$\mathbb{E}[N_{M}(ds,d\omega)]/(B_{0}(d\omega)ds)=\hat{\rho}_{(\beta+M)r}(s).$$
Hence $\mu_{M}$ is $\mathrm{CRM}(\hat{\rho}_{(\beta+M)r}B_{0}),$ $Z_{M+1}|Z_{1},\ldots, Z_{M}$ is such that
$$\tilde{Z}_{M+1}\sim\mathrm{IBP}(NB(r,s), \hat{\rho}_{(\beta+M)r}B_{0})$$ and the $A_{\ell}$ are conditionally $\mathrm{NB}(r,J_{\ell})$ variables.

\subsection{Describing the marginal distribution of  $Z$ via a method of decompositions}
The results above provide a general framework to describe generative processes analogous to the Indian Buffet sequential latent feature scheme. What remains is to describe the marginal distributions of $Z_{1}\sim \mathrm{IBP}(A,\rho B_{0}),$ and hence $\tilde{Z}_{M+1}\sim \mathrm{IBP}(A,\rho_{M} B_{0}),$ which makes up the un-sampled part of $Z_{M+
1}|Z_{1},\ldots, Z_{M}.$ 

Since $Z:=Z_{1}\overset{d}=\sum_{k=1}^{\infty}A_{k}\delta_{\tilde{\omega}_{k}}$ is composed of $(A_{k},\tau_{k},\tilde{\omega}_{k})$ the theory of marked Poisson point processes tells us that the unconditional Laplace functional of $$Z(f)=\sum_{k=1}^{\infty}A_{k}f(\omega_{k})$$ is given by $\mathbb{E}[{\mbox e}^{-Z(f)}]={\mbox e}^{-\Phi(f)},$ where 
$$
\Phi(f)=\int_{\Omega}\int_{\Theta}(1-e^{-af(\omega)})I_{\{a\neq 0\}}Q_{A}(da|\omega)B_{0}(d\omega)
$$
where
$
Q_{A}(da|\omega)=\left[\int_{0}^{\infty}G_{A}(da|s)\rho(s|\omega)ds\right],
$
and the distribution of the $A_{k}$ is determined by the measure $I_{a\neq 0}Q_{A}(da|\omega).$ This is clear since at $a=0,$ $(1-e^{-af(\omega)})$ is zero. 
If we assume that $\rho(s|\omega)$ corresponds to an infinite activity process then $\int_{0}^{\infty}\rho(s|\omega)ds=\infty.$ This means essentially that if $G_{A}(da|s)$ does not assign mass to $0,$ $I_{\{a\neq 0\}}Q_{A}(da|\omega)$ is not a finite measure. So for instance the $(A_{k}),$ after marginalizing out $(\tau_{k}),$ could be the jumps of a gamma process.  For concreteness we pause to give an example of that case. 

Suppose that $S_{\alpha}$ is a positive stable random variable of index $0<\alpha<1$ with log Laplace exponent evaluated at $\beta$ equal to $\beta^{\alpha}$ then set 
$$
G_{A}(da|t)=\mathbb{P}(S^{-\alpha}_{\alpha}t\in da){\mbox { and }}\rho(t)=\alpha t^{-1}{\mbox e}^{-\beta t},
$$
hence $\rho$ is the Levy density of a gamma process with shape $\alpha$ and scale $1/\beta.$ Then from Aldous and Pitman~\citep[eq.(33)]{AldousPitR}there is the identity
$$
\alpha t^{-1}{\mbox e}^{-\beta t}\mathbb{P}(S^{-\alpha}_{\alpha}t\in da)=
a^{-1}\mathbb{P}(a^{1/\alpha}S_{\alpha}\in dt){\mbox e}^{-t\beta}
$$
which shows that in this instance
$$
Q_{A}(da)=\int_{0}^{\infty}G_{A}(da|t)\rho(t)dt=a^{-1}{\mbox e}^{-a\beta^{\alpha}}da.
$$
That is to say the marginal distribution of $Z$ corresponds to a Gamma process. 

Now if $1-\pi_{A}(s)=\mathbb{P}(A=0|s)>0,$ as we have assumed, denotes the probability that $A=0$ under $G_{A}$ then if $\varphi(\omega)=\int_{0}^{\infty}\pi_{A}(s)\rho(s|\omega)ds<\infty$ the $A_{k}|\omega_{k}$ are independent with distribution 
$$
\tilde{Q}_{A}(da|\omega_{k})=I_{\{a\neq 0\}}Q_{A}(da|\omega_{k})/\varphi(\omega_{k}).
$$
Of course if $\rho$ is homogeneous the $A_{k}$ would be iid $\tilde{Q}_{A}.$ In this case set $\varphi(\omega):=\varphi<\infty$ then it follows that the marginal process $Z$ can be represented as $Z=\sum_{k=1}^{\xi(\varphi)}X_{k}\delta_{\tilde{\omega}_{k}}$ where $\xi(\varphi)$ is a Poisson random variable with mean $\varphi$ and the $X_{i}$ are iid  $\tilde{Q}_{A}.$ 
However even in this ideal situation. it will be difficult to sample directly from $\tilde{Q}_{A}$ for arbitrary $\rho.$ We now describe the marginal distribution of $Z$ utilizing an infinite sequence of iid pairs of variables $((H_{i},X_{i})).$
\begin{prop}
\label{Hdecomp} 
Let $Z|\mu$ be $\mathrm{IBP}(G_A|\mu)$ where $\mu$ is a $\mathrm{CRM}(\rho B_{0})$ with $\rho(s)$ a homogeneous L\'evy density. Then the $Z$ is said to have an $\mathrm{IBP}(A,\rho B_{0})$ marginal distribution. Suppose that $\varphi:=\varphi(\rho|\pi_{A})=\int_{0}^{\infty}\pi_{A}(s)\rho(s)ds<\infty,$ then there exists a sequence of iid pairs $((H_{i},X_{i})),$ which we refer to as an $\mathrm{IBP}(A,\rho )$ process, such that 
$$
Z\overset{d}=\sum_{i=1}^{\xi(\varphi)}X_{i}\delta_{\tilde{\omega}_{i}}
$$
where $\xi(\varphi)$ is a Poisson random variable with mean $\varphi$ independent of  $((H_{i},X_{i})).$ The pair have the following distributional properties:
\begin{enumerate}
\item[(i)] $X_{i}|H_{i}=s$ has distribution, not depending on $\rho,$  
$$
\frac{
I_{\{a\neq 0\}}G_{A}(da|s)}{\pi_{A}(s)},
$$
which is equivalent to the conditional distribution of $A|A\neq 0$ under $G_{A}(\cdot|s),$
and $H_{i}$ has marginal density
$$
\frac{\pi_{A}(s)\rho(s)}{\varphi(\rho|\pi_{A})}.
$$
\item[(ii)] The marginal distribution of $X_{i}$ is 
$$
\mathbb{P}(X_{i}\in da)=\frac{
I_{\{a\neq 0\}}\int_{0}^{\infty}G_{A}(da|s)\rho(s)ds}{\varphi(\rho|\pi_{A})}
$$
and the distribution of $H_{i}|X_{i}=a$ is given by
$$
\mathbb{P}(H_{i}\in ds|X_{i}=a)=\frac{
I_{\{a\neq 0\}}G_{A}(da|s)\rho(s)ds}{\int_{0}^{\infty}G_{A}(da|v)\rho(v)dv}.
$$
\end{enumerate}
\end{prop}
\begin{proof} The result is straightforward as $\varphi<\infty,$ coupled with the results we discussed previously for marked Poisson processes, allows one to manipulate a joint distribution proportional to 
$$
I_{\{a\neq 0\}}G_{A}(da|s)\rho(s)ds.
$$
\end{proof}
 
\section{Essentials for implementation}
As mentioned previously one does not need to understand the PPC in order to use the results that have been obtained. The ingredients for that are listed below. 
\subsection{Steps for describing the posterior distribution, marginal distributions and related quantities}
\label{posterioringredients}
\begin{enumerate}
\item[(i)]Specify $G_{A}(\cdot|s)$
\item[(ii)] Identify $\pi_{A}(s)=\mathbb{P}(A\neq 0|s)$
\item[(iii)]For each $M=0,1,2,\ldots,$ $f_{M}(s,\omega)=-M\log(1-\pi_{A}(s))$ implying 
$$
\nu_{f_{M}}(s,\omega)={[1-\pi_{A}(s)]}^{M}\rho(s|\omega)B_{0}(d\omega)ds
$$
and
$$
\Psi(f_{M})=\int_{\Omega}\int_{0}^{\infty}({[1-\pi_{A}(s)]}^{M})\rho(s|\omega)dsB_{0}(d\omega)
$$
Hence for each $M=0,1,2,\ldots$ define $\rho_{M}(s)={[1-\pi_{A}(s)]}^{M}\rho(s|\omega).$
\item[(iv)]
Use this to specify ${N}_{M}\sim \mathrm{PRM}(\rho_{M}B_{0}),$ ${\mu}_{M}\sim \mathrm{CRM}(\rho_{M}B_{0})$ and $\tilde{Z}_{M+1}\sim \mathrm{IBP}(A,\rho_{M}B_{0}).$
\item[(v)]Identify
$$
h_{i,\ell}(s)={\left[\frac{G_{A}(da_{i,\ell}|s_{\ell})}{
1-\pi_{A}(s_{\ell})}\right]}^{\indic_{\{a_{i,\ell}\neq 0\}}},
$$
and look for simplifications of $\prod_{i=1}^{M}h_{i,\ell}(s),$ which certainly happens when $A$ is Bernoulli, Poisson or Negative-Binomial.
\item[(vi)] Use this to get the distribution of $(J_{\ell})$ and marginal of $(Z_{1},\ldots,Z_{M})$ specified by (\ref{marginalJ}) and (\ref{margZ}). This gives the distribution of $(A_{\ell})$ specified in Proposition~\ref{Zproposition}. 
\item[(vii)] Note multiplicities (counts) seen in the cases of Bernoulli, Poisson and Negative-Binomial arise through those choices of $G_{A}$ and manifest themselves through simplifications of $h_{i,\ell}.$
\end{enumerate}
\subsection{Ingredients for sequential generalized Indian buffet schemes using Proposition~\ref{Hdecomp}}
\label{IBPingredients}
\begin{enumerate}
\item[(i)]For each $M=0,1,2,\ldots$ define $\rho_{M}(s)={[1-\pi_{A}(s)]}^{M}\rho(s).$
\item[(ii)] Calculate and check, 
$$
\varphi_{M+1}(\rho)=\int_{0}^{\infty}\pi_{A}(s)\rho_{M}(s)ds<\infty.
$$
\item[(iii)]Then if $\tilde{Z}_{M+1}$ is $IBP(A,\rho_{M}B_{0}),$ 
$$
\tilde{Z}_{M+1}\overset{d}=\sum_{k=1}^{\xi(\varphi)}X_{k}\delta_{\tilde{\omega}_{k}}
$$
where $\varphi:=\varphi_{M+1}(\rho),$ and $\xi(\varphi)$ is a Poisson$(\varphi)$ random variable, independent of $(X_{k})$ which are taken from a $\mathrm{IBP}(A,\rho_{M})$ process $((X_{i},H_{i}))$
\item[(iv)] $X_{i}|H_{i}=s$ has distribution, not depending on $\rho,$  
\begin{equation}
\label{Xcond}
\frac{
I_{\{a\neq 0\}}G_{A}(da|s)}{\pi_{A}(s)}
\end{equation}
and $H_{i}$ has marginal density
\begin{equation}
\label{Hmarg}
\frac{\pi_{A}(s)\rho_{M}(s)}{\varphi_{M+1}(\rho)}
\end{equation}
\item[(v)] The marginal distribution of $X_{i}$ is 
$$
\mathbb{P}(X_{i}\in da)=\frac{
I_{\{a\neq 0\}}\int_{0}^{\infty}G_{A}(da|s)\rho_{M}(s)ds}{\varphi_{M+1}(\rho)}
$$
and the distribution of $H_{i}|X_{i}=a$ is given by
$$
\mathbb{P}(H_{i}\in ds|X_{i}=a)=\frac{
I_{\{a\neq 0\}}G_{A}(da|s)\rho_{M}(s)ds}{\int_{0}^{\infty}G_{A}(da|v)\rho_{M}(v)dv}.
$$
\end{enumerate}.
\section{ Generalized Indian Buffet Processes: The sequential generative process for  $\mathrm{IBP}(A,\rho B_{0})$}
We can now use sections \ref{posterioringredients} and \ref{IBPingredients} to describe the sequential generative process for $\mathrm{IBP}(A,\rho,B_{0}),$ in the homogeneous case. That is to say how to sample from $Z_{1},$ and subsequently $Z_{M+1}|Z_{1},\ldots, Z_{M}.$
Since under $G_{A}$ every matrix entry  value for a feature can take a wider range of values, one needs to give this a proper interpretation. This is mostly left to the reader, but here we shall naively say that a customer selects a dish many times, or gives a scoring to that dish.  At any rate we use the basic IBP metaphor of people sequentially entering an Indian Buffet restaurant. Recall again that $\varphi(\rho)=\int_{0}^{\infty}\pi_{A}(s)\rho(ds).$

\begin{enumerate}
  \item Customer 1 selects dishes and gives them scores according to a $\mathrm{IBP}(A,\rho B_{0})$ process. This is done precisely as follows:
  \begin{enumerate}
\item[(i)] Draw a $\mathrm{Poisson}(\varphi_{1}(\rho))=J$ number of variables.
\item[(ii)] Draw $((\omega_{1},H_{1}),\ldots (\omega_{J},H_{J}))$ iid from $B_{0}$ and the distribution for $H_{i}$ specified by (\ref{Hmarg}), with $M=0$
\item[(iii)] Draw $X_{i}|H_{i}$ following (\ref{Xcond}), for $i=1,\ldots. J.$ Note if it is straightforward to sample directly from the marginal distribution of $X_{i},$ then one may bypass sampling $H_{i}$
  \end{enumerate}
  \item After $M$ customers have chosen collectively $\ell=1,\ldots,K$ distinct dishes, and assigned their scores, customer $M+1$ selects each dish $\omega_{\ell}, \ell=1,\ldots,K,$ and assigns respective scores, according to the distribution of $A_{\ell},$ where $A_{\ell}|J_{\ell}$ has conditional distribution $G_{A}(da|J_{\ell}).$  in particular the chance that $A_{\ell}|J_{\ell}$ takes the value zero is $1-\pi_{A}(J_{\ell}).$ The  distribution of $(J_{\ell})$ is specified in~(\ref{marginalJ}). 
  \item Customer M+1 also chooses and scores new dishes according to a $\mathrm{IBP}(A,\rho_{M}B_{0}),$ process $\tilde{Z}_{M+1}.$ This follows the same scheme as customer $1$ with $\rho_{M}(s)={[1-\pi_{A}(s)]}^{M}\rho(s)$ in place of $\rho$ and details provided in section \ref{IBPingredients}.
\end{enumerate}
We next provide more details for $(H_{i},X_{i})$ in the Poisson and Negative-Binomial cases. 
\subsection{$\mathrm{IBP}(\mathrm{Poi}(b\lambda), \rho)$ processes} 
In the Poisson IBP setting the random $\mathrm{IBP}(\mathrm{Poi}(b\lambda), \rho)$ process $((H_{i},X_{i}))$ actually appears in  Pitman~\cite{Pit97}.
Here again for simplicity we shall assume $\rho(\lambda)$ is a homogeneous Levy density not depending on $\omega.$ The extension is obvious. In this case $X_{i}$ takes values $j=1,\ldots,\infty$ with 
$$
\mathbb{P}(X_{i}=j)=\frac{b^{j}\int_{0}^{\infty}\lambda^{j}{\mbox e}^{-\lambda b}\rho(\lambda)d\lambda}{j!\psi(b)}.
$$
Note that $\pi_{A}(\lambda)=\mathbb{P}(A>0|\lambda)=1-{\mbox e}^{-\lambda b},$ and 
$$
\varphi:=\psi(b)=\int_{0}^{\infty}(1-{\mbox e}^{-b\lambda})\rho(\lambda)d\lambda
$$
is the relevant total mass. In addition, $\kappa_{j}(\tilde{\rho}_{b})=\int_{0}^{\infty}\lambda^{j}{\mbox e}^{-\lambda b}\rho(\lambda)d\lambda$ is the $j$-th cumulant of an infinitely divisible random variable $T_{b}$ with density ${\mbox e}^{-[bs-\psi(b)]}f_{T}(s).$
Now conditionally given all $(X_{k})$ the $H_{i}$ are conditionally independent with,
$$
\mathbb{P}(H_{i}\in ds|X_{i}=j)=\frac{s^{j}{\mbox e}^{-sb}\rho(s)}{\kappa_{j}(\tilde{\rho}_{b})}
$$
The unconditional distribution of each $H_{i}$ is 
$$
\mathbb{P}(H_{i}\in ds)=(1-{\mbox e}^{-bs})\rho(s))/\psi(b).
$$
The variable $X_{i}$ can be obtained and interpreted as follows:$X_{i}|H_{i}=\lambda$ is a random variable that is equivalent in distribution to a Poisson random variable  with intensity $(b\lambda)$ conditioned on  the event that the Poisson variable is at least $1.$ That is for $j=1,2,\ldots$
$$
\mathbb{P}(X_{i}=j|H_{i}=\lambda)=\frac{b^{j}\lambda^{j}{\mbox e}^{-\lambda b}}{j!(1-{\mbox e}^{-b\lambda})}.
$$
Lastly it is easy to see that $\mathrm{E}[X_{i}]=b\int_{0}^{\infty}s\rho(s)ds/\psi(b),$ which is possibly infinite. 
\begin{rem}For sampling $\tilde{Z}_{M+1}$ replace $\rho(s)$ with ${\mbox e}^{-bMs}\rho(s).$
\end{rem}
\subsubsection{Examples, gamma, stable, and generalized gamma cases: Part 2}
We now look at the special cases of the Gamma and generalized gamma processes considered earlier. 

If for $0<\alpha<1,$ $T$ is a $\alpha$ stable random variable  with density denoted as $f_{\alpha},$ $\psi_{\alpha}(b)=b^{\alpha}$ and $\rho_{\alpha}(s):=\alpha s^{-\alpha-1}/\Gamma(1-\alpha)$ Then  for any $b>0,$
$$
\mathbb{P}(X_{i}=j)=\frac{\alpha \Gamma(j-\alpha)}{j!\Gamma(1-\alpha)},
$$
this is sometimes referred to as Sibuya's distribution, and $H_{i}|X_{i}=j$ is a Gamma$(j-\alpha,1/b)$ random variable.  Furthermore 
$\mathbb{E}[X_{i}]=\infty.$ See Pitman~\cite{Pit97} for more details related to this case.

If $T$ is a generalized gamma random variable with density ${\mbox e}^{-[t\zeta-\zeta^{\alpha}]}f_{\alpha}(t),$ $\rho_{\alpha,\zeta}(s):=\alpha s^{-\alpha-1}{\mbox e}^{-s\zeta}/\Gamma(1-\alpha),$ and $\psi_{\alpha,\zeta}(b)=[(\zeta+b)^{\alpha}-\zeta^{\alpha}]$ then 
$$
\mathbb{P}(X_{i}=j)=\frac{{(b+\zeta)}^{\alpha-j}b^{j}}
{\psi_{\alpha,\zeta}(b)}\frac{\alpha \Gamma(j-\alpha)}{j!\Gamma(1-\alpha)}
$$
and $H_{i}|X_{i}=j$ is Gamma $(j-\alpha, 1/(b+\zeta)).$
$\mathbb{E}[X_{i}]=b\alpha\zeta^{\alpha-1}/\psi_{\alpha,\zeta}(b).$
\begin{rem}
In reference to sampling $\tilde{Z}_{M+1}$ note that 
$$
\rho_{M}(s)=\rho_{\alpha,bM+\zeta}(s)=\alpha s^{-\alpha-1}{\mbox e}^{-s(bM+\zeta)}/\Gamma(1-\alpha).
$$
and one simply replaces $\zeta$ in the above quantities with $\zeta+bM.$ In the stable case, customer $M+1$ 
selects a $\mathrm{Poisson}(b^{\alpha}[{(M+1)}^{\alpha}-M^{\alpha}])$
number of new dishes and draws the corresponding $X_{i}$ from
$$
\mathbb{P}(X_{i}=j)=\frac{{(1+M)}^{\alpha-j}}
{[{(M+1)}^{\alpha}-M^{\alpha}]}\frac{\alpha \Gamma(j-\alpha)}{j!\Gamma(1-\alpha)}.
$$
\end{rem}
Suppose that $T$ is now a gamma random variable with shape parameter $\theta$ and scale $1.$ Then $\rho(s)=\theta s^{-1}{\mbox e}^{-s},$ $\psi(b)=\theta \log(1+b)$ and 
$$
\mathbb{P}(X_{i}=j)=\frac{b^{j}{(1+b)}^{-j}}{j\log(1+b)}
$$
and $H_{i}|X_{i}=j$ is Gamma $(j, 1/(1+b)).$ It is curious that the random variables do not depend on $\theta.$ $\mathbb{E}[X_{i}]=b/[\log(1+b)].$ Note in this case the $X_{i}$ are iid with a discrete Logarithmic distribution with parameter $b/(1+b)$ and it follows that
$$
Z_{1}(\Omega)\overset{d}=\sum_{i=1}^{\xi(\theta \log(1+b))}X_{i}\sim \mathrm{NB}(\theta,b/(1+b)),
$$
and $\mathbb{E}[Z_{1}(\Omega)]=\theta b.$
Note again that the gamma case is the case considered by Titsias.
\begin{rem}In order to sample $\tilde{Z}_{M+1}$ in the gamma case, note that $\rho_{M}(s)=\theta
s^{-1}{\mbox e}^{-(bM+1)s}.$
Hence $\psi(b)=\theta\log(1+b/(bM+1))$ and 
$$
\mathbb{P}(X_{i}=j)=\frac{b^{j}{(1+b(M+1))}^{-j}}{j\log(1+b/(bM+1))}
$$
which is a discrete Logarithmic distribution with parameter $b/(1+b(M+1)).$ Hence it follows that $\tilde{Z}_{M+1}(\Omega)$ has an 
$\mathrm{NB}(\theta,b/(1+b(M+1
+)))$ distribution.
\end{rem}

\subsection{$\mathrm{IBP}(\mathrm{NB}(r,p), \rho)$ processes for generating Z in the Negative Binomial case}
We now look at the Negative Binomial setting. For this class of models the relevant $\tilde{Q_{A}}$ measure is a proper distribution determined by the joint measure, 
$$
\indic_{\{a\neq 0\}}{a+r-1\choose a}p^{a}{(1-p)}^{r}\rho(p).
$$
The probability that a Negative Binomial random variable is greater than $0$ is $[1-(1-p)^{r}]$ and we see that the relevant total mass is 
$$
\varphi=\phi(r)=\int_{0}^{1}[1-(1-p)^{r}]\rho(p)dp=r\int_{0}^{1}p{(1-p)}^{r-1}\rho(p)dp.
$$

We now say $((H_{i},X_{i}))$ follows a $\mathrm{IBP}(\mathrm{NB}(r,p), \rho)$ process if $X_{i}|H_{i}=p$ has distribution, for $j=1,2,\ldots,$ 
\begin{equation}
\label{NBcond}
\mathbb{P}(X_{i}=j|H_{i}=p)=\frac{{j+r-1\choose j}p^{j}{(1-p)}^{r}}{[1-{(1-p)}^{r}]}
\end{equation}
and $H_{i}$ has marginal density
\begin{equation}
\label{NBHmarginal}
\frac{{[1-{(1-p)}^{r}}]\rho(p)}{\phi(r)}.
\end{equation}
Notice that the distribution of $X_{i}|H_{i}=p$ corresponds to the  distribution of a Negative Binomial random variable conditioned so it takes values greater than $0.$ Notice also that here, as in the Poisson case (and indeed any $A$),  this distribution does not depend on $\rho.$ Obviously the marginal distribution of $X_{i}$ is 
\begin{equation}
\label{NBdist}
\mathbb{P}(X_{i}=j)=\frac{{j+r-1\choose j}\int_{0}^{1}p^{j}{(1-p)}^{r}\rho(p)dp}{\phi(r)}=\frac{{j+r-1\choose j}\kappa_{j,r}(\rho)}{\phi(r)}
\end{equation}
and $H_{i}|X_{i}=j$ is
\begin{equation}
\label{NBHconddist}
\mathbb{P}(H_{i}\in dp|X_{i}=j)=\frac{p^{j}{(1-p)}^{r}\rho(p)dp}{\kappa_{j,r}(\rho)}.
\end{equation} 
One can check that $\mbox{E}[X_{i}]=r\int_{0}^{1}p{(1-p)}^{-1}\rho(p)dp/\phi(r),$ which is certainly not always finite.
\begin{rem}For sampling $\tilde{Z}_{M+1}$ replace $\rho(s)$ with ${(1-s)}^{Mr}\rho(s).$
\end{rem}
\subsection{Surprise, some beta process priors are explosive in the Negative Binomial setting} As was mentioned much of the work done so far with the Negative Binomial case has involved the use of the beta process. It is quite straightforward to obtain explicit expressions based on our results so such calculations  are in fact omitted. However we have noticed a surprising fact in this case which we find is worthwhile to mention. Consider a simple beta process with intensity
$$
\rho_{\theta,\beta}(p)=\theta p^{-1}{(1-p)}^{\beta-1}.
$$
which is well defined for any $\beta>0.$
Then applying (\ref{NBdist}) it is easy to see that the marginal distribution of $X_{i}$ would correspond to the result given in~\cite{Croy}. However what seems not to have been pointed out is the following. Note that if $Z\sim IBP(
\mathrm{NB}(r,p), \rho B_{0})$ then 
$$
\mathbb{E}[Z(\Omega)]=\varphi E[X_{1}]=r\int_{0}^{1}p{(1-p)}^{-1}\rho(p)dp.
$$
If we use $\rho_{\theta,\beta}(p)$ then for $Z\sim IBP(
\mathrm{NB}(r,p), \rho_{\theta,\beta} B_{0}),$
$$
\mathbb{E}[Z(\Omega)]=\varphi E[X_{1}]=\theta r\int_{0}^{1}{(1-p)}^{\beta-2}dp.
$$
So it follows that if $\beta\leq 1$ then $\mathbb{E}[Z(\Omega)]=\infty.$ From section 4 note this still means that a customer selects a Poisson $(\varphi),$ $\varphi=\phi(r),$ number of dishes, but in these cases gives very high scores to each dish, which is reflected by $\mathbb{E}[X_{1}]=\infty.$
\section{Constrained Multivariate Priors and IBP generalizations}
The examples of Poisson, Bernoulli and Negative Binomial processes we presented in the previous section have all appeared in the recent literature, albeit with respect to more confined choices for $\mu.$ There is certainly interest in possible multivariate extensions of such processes. The purpose of this section is to demonstrate that from a technical viewpoint the $\mathrm{PPC}$ is well-equipped to easily handle this. This importantly would allow one to focus on modeling issues and interpretations that arise in a more intricate multidimensional setting rather than be encumbered by the calculus of random measures. We present results for a simple multinomial setting then a very general multivariate setting. The results although stated in a brief form are obtained in a formal manner as the PPC makes such arguments rather transparent. The joint distributions are very similar to the univariate cases. 

We discuss the case of multivariate CRM's which have suitable constraints to handle for instance multinomial extensions of the IBP, and provide an example of an IBP-like model. 
Lets first say a few words about multivariate Levy processes with positive jumps. For $\mathbf{x}_{q}=(x_{1},x_{2},\ldots, x_{q}),$ let $\rho_{q}(\mathbf{x})$ denote a sigma-finite function concentrated on $\mathbb{R}^{q}_{+}/\{\mathbf{0}\}$ then following Barndorff-Nielsen, Pedersen and Sato~\citep[Proposition 3.1]{BNPS} there exists a multivariate CRM $\mu_{0}:=(\mu_{j},j=1,\dots,q)$ with jumps specified by $\rho_{q}$ if 
$\int_{\mathbb{R}^{q}_{+}}\min(x_{\cdot},1)\rho_{q}(\mathbf{x})d\mathbf{x}<\infty,$ where $x_{\cdot}=\sum_{j=1}^{q}x_{j}.$ In the next section we  shall focus on models with the following constraints 
$$\mathcal{S}_{q}={\{\mathbf{s}_{q}:s_{j}>0,s_{\cdot}=\sum_{j=1}^{q}s_{j}<1\}}.$$ 

However in the final section we will not specify this. Naturally we shall assume a common base measure $B_{0}(d\omega)$ representing again features. So it follows that one can write $\mu_{j}=\sum_{k=1}^{\infty}p_{j,k}\delta_{\tilde{\omega}_{k}}.$ Furthermore 
$\mu_{\cdot}=\sum_{j=1}^{q}\mu_{j}=\sum_{k=1}^{\infty}p_{\cdot,k}\delta_{\tilde{\omega}_{k}},$ for  $p_{\cdot,k}=\sum_{j=1}^{q}p_{j,k}.$ The multivariate CRM $\mu_{0}$ is constructed from  $N$ a Poisson random measure on $(\mathbb{R}^{q}_{+},\Omega)$
with intensity $\nu(d\mathbf{s}_{q},d\omega)=\rho_
{q}(d\mathbf{s}_{q}|\omega)B_{0}(d\omega),$ and can be represented as $N:=\sum_{k=1}^{\infty}\delta_{\mathbf{p}_{q,k},\tilde{\omega}_{k}},$ where $\mathbf{p}_{q,k}=(p_{1,k},\ldots,p_{q,k}).$ 

\begin{rem}
For example if $\rho_{2}(x_{1},x_{2})=\theta{(x_{1}+x_{2})}^{-2}\indic_{\{0<x_{1}+x_{2}<1\}}$ then it follows making the transformation $s=x_{1}+x_{2}$ yields the marginal Levy density for $\mu_{\cdot}$ to be $\theta s^{-1}\indic_{\{0<s<1\}}.$ See\cite{KimY2} and section~\ref{bD} below for a more general class of models.
\end{rem}
\subsection{Simple Multinomial case}\label{Multi}
Now conditional on $\mu_{0},$ for each fixed $k$ and $i$ let $b^{(i)}_{0,k}:=(b^{(i)}_{j,k}, j=1,\ldots, q)$
denote an independent Multinomial $(1,(p_{j,k}), 1-p_{.,k}):=\mathsf{M}(1,\mathbf{p}_{q,k})$ vector. That is for each fixed $(i,k)$ the joint probability mass function is given by,
$$
\left[\prod_{j=1}^{q}p^{b^{(i)}_
{j,k}}_{j,k}\right](1-p_{\cdot,k})^{1-b^{(i)}_{\cdot,k}}
$$
where at most one of the terms in $b^{(i)}_{0,k}:=(b^{(i)}_{j,k}, j=1,\ldots, q)$ is one and the others are $0.$ 
Then we can define a vector valued process ${Z}^{(i)}_{0}:=(Z^{(i)}_{1},\ldots,Z^{(i)}_{q}),$ where $Z^{(i)}_{j}=\sum_{k=1}^{\infty}b^{(i)}_{j,k}\delta_{\tilde{\omega}_{k}},$ and $Z^{(i)}_{\cdot}=\sum_{j=1}^{q}Z^{(i)}_{j}$ is an IBP Bernoulli process. We say that ${Z}^{(1)}_{0},\ldots,{Z}^{(M)}_{0}|\mu_{0}$
are iid $\mathrm{IBP}(\mathsf{M}(1,\mathbf{p}_{q})|\mu_{0}),$ and marginally ${Z}^{(1)}_{0}$ is $\mathrm{IBP}(\mathsf{M}(1,\mathbf{p}_{q}),\rho_{q}B_{0}).$
 Now based on $M$ conditionally iid processes the likelihood can be written as
$$
\prod_{i=1}^{M}\prod_{k=1}^{\infty}\left[\prod_{j=1}^{q}p^{b^{(i)}_
{j,k}}_{j,k}\right](1-p_{\cdot,k})^{1-b^{(i)}_{\cdot,k}}=
\left[\prod_{l=1}^{K}\prod_{j=1}^{q}{\left(\frac{p_{j,\ell}}{1-p_{\cdot,\ell}}\right)}^{c_{j,\ell,M}}\right]{\mbox e}^{-M\sum_{j=1}^{\infty}[-\log(1-p_{\cdot,j})]}
$$
where $c_{j,\ell,M}=\sum_{i=1}^{M}b^{(i)}_{j,\ell}$ and $c_{\ell,M}=\sum_{j=1}^{q}c_{j,\ell,M}.$
Now let $N$ denote a Poisson random measure on $(\mathbb{R}^{q}_{+},\Omega)$
with intensity $\nu(d\mathbf{s}_{q},d\omega)=\rho_
{q}(d\mathbf{s}_{q}|\omega)B_{0}(d\omega).$ Now define $g_{j}(\mathbf{s}_{q})=s_{j}/(1-s_{\cdot})$
for $j=1,\ldots q.$ Despite the added complexity the likelihood is similar in form to the simple Bernoulli process case and one can immediately conclude that a joint distribution of  $N,(\mathbf{J}_{q,\ell}\in \mathcal{S}_{q}, \ell=1,\ldots,K),(Z^{(1)}_{0},\ldots, Z^{(M)}_{0})$ is given by, 
$$
\mathcal{P}(dN|\nu_{f},(\mathbf{s}_{q},\omega))\left[\prod_{\ell=1}^{K}\mathbb{P}(\mathbf{J}_{q,\ell}\in d\mathbf{s}_{q,\ell})\right]\mathbb{P}(Z^{(1)}_{0},\ldots, Z^{(M)}_{0})
$$
where now $f(\mathbf{s}_{q},\omega)=-M\log(1-s_{\cdot})$ and hence 
\begin{equation}
\nu_{f}(d\mathbf{s}_{q},d\omega)={(1-s_{\cdot})}^{M}\nu(d\mathbf{s}_{q},d\omega)=
{(1-s_{\cdot})}^{M}
\rho_
{q}(d\mathbf{s}_{q}|\omega)B_{0}(d\omega),
\label{multinu}
\end{equation}
and the joint distribution of $(\mathbf{J}_{q,\ell}\in \mathcal{S}_{q},\ell=1,\ldots,K),(Z^{(1)}_{0},\ldots, Z^{(M)}_{0})$ is given by, 
$$
{\mbox e}^{-\phi(M)}\prod_{\ell=1}^{K}\left[\prod_{j=1}^{q}{[g_{j}(\mathbf{s}_{q,\ell})]}^{c
_{j,\ell,M}}\right]{(1-s_{\cdot,\ell})}^{M}\rho_{q}(d\mathbf{s}_{q,\ell}|\omega_{\ell})B_{0}(d\omega_{\ell}),
$$
where $\phi(M)=\int_{\Omega}\int_{R^{q}_{+}}(1-[1-s_{\cdot}]^{M})\rho_{q}(d\mathbf{s}_{q}|\omega)B_{0}(d\omega).$
Again applying Bayes rule, this is enough to conduct a posterior analysis in parallel to the univariate case.
Set 
$$
\kappa_{\mathbf{c}_{\ell,M}}(\rho_{q,M}|\omega_{\ell})=
\int_{\mathbf{S}_{q}}\left[\prod_{j=1}^{q}{p_{j}}^{c
_{j,\ell,M}}\right]{(1-p_{\cdot})}^{M-c_{\ell,M}}\rho_{q}(d\mathbf{p}_{q}|\omega_{\ell})
$$
where $c_{\ell,M}=\sum_{j=1}^{q}c_{j,\ell.M},$ $\mathbf{c}_{\ell,M}=(c_{1.\ell,M},\ldots,c_{q,\ell,M}).$

\begin{prop}\label{multiprop}Suppose that ${Z}^{(1)}_{0},\ldots,{Z}^{(M)}_{0}|\mu_{0}$
are iid $\mathrm{IBP}(\mathsf{M}(1,\mathbf{p}_{q})|\mu_{0}),$  as described above in section~\ref{Multi}. Then,
\begin{enumerate}
\item[(1)]the posterior distribution of $N|{Z}^{(1)}_{0},\ldots,{Z}^{(M)}_{0},$ is equivalent in distribution to the random measure
$$
N_{M}+\sum_{\ell=1}^{K}\delta_{\mathbf{J}_{q,\ell},
\omega_{\ell}}$$
where $N_{M}$ is a $\mathrm{PRM}(\rho_{q,M}B_{0})$ with mean intensity described in 
(\ref{multinu}) and conditionally independent of $N_{M},$ 
the $(\mathbf{J}_{q,\ell})$ are an independent collection of random vectors such that $\mathbf{J}_{q,\ell}=(J_
{1,\ell},\ldots, J_{q,\ell})$ has joint density,
$$
\frac{\left[\prod_{j=1}^{q}{[s_{j,\ell}]}^{c
_{j,\ell,M}}\right]{(1-s_{\cdot,\ell})}^{M-c_{\ell,M}}\rho_{q}(d\mathbf{s}_{q,\ell}|\omega_{\ell})}{\kappa_{\mathbf{c}_{\ell,M}}(\rho_{q,M}|\omega_{\ell})},
$$
where $c_{\ell,M}=\sum_{j=1}^{q}c_{j,\ell.M},$ $\mathbf{c}_{\ell,M}=(c_{1.\ell,M},\ldots,c_{q,\ell,M}).$
\item[(2)]It follows that ${Z}^{(M+1)}_{0}|{Z}^{(1)}_{0},\ldots,{Z}^{(M)}_{0}$  can be represented in terms of   ${\tilde{Z}}^{(M+1)}_{0}:=(\tilde{Z}^{(M+1)}_{1},\ldots,\tilde{Z}^{(M+1)}_{q}),$ which is determined by $N_{M}$ call it an $\mathrm{IBP}(\mathsf{M}(1,\mathbf{p}_{q}),\rho_{q,M}B_{0})$ vector, and $\ell=1,\ldots, K$  vectors $(b^{(\ell)}_{0}), $ where for each fixed $\ell,$ $b^{(\ell)}_{0}|\mathbf{J}_{q,\ell}=\mathbf{s}_{q}$ has distribution $\mathsf{M}(1,\mathbf{s}_{q}).$ Specifically, component-wise $\tilde{Z}^{(M+1)}_{0}$
can be represented in distribution as $$\tilde{Z}^{(M+1)}_{j}+\sum_{\ell=1}^{K}b^{(\ell)}_{j}\delta_{{\omega}_{\ell}},$$ where $b^{(\ell)}_{j}$ is the j-th component of $b^{(\ell)}_{0}.$
\item[(3)]Marginally for each $\ell,$ $b^{(\ell)}_{0}$ is $\mathsf{M}(1,\mathbf{r}_{q,\ell}),$ where $\mathrm{r}_{q,\ell}=(r_{1,\ell},\ldots,r_{q,\ell})$ defined for $k=1,\ldots,q,$ as
$$
r_{k,\ell}=\mathbb{E}[J_{k,\ell}|\mathbf{c}_{\ell,M}]=
\frac{\int_{\mathbf{S}_q}s_{k}\left[\prod_{j=1}^{q}{[s_{j}]}^{c
_{j,\ell,M}}\right]{(1-s_{\cdot})}^{M-c_{\ell,M}}\rho_{q}(d\mathbf{s}_{q}|\omega_{\ell})}{\kappa_{\mathbf{c}_{\ell,M}}(\rho_{q,M}|\omega_{\ell})}.
$$
\item[(4)]Hence $b^{(\ell)}_{\cdot}=\sum_{j=1}^{q}b^{(\ell)}_{j}$ is $\mathrm{Bernoulli}(\sum_{j=1}^{q}r_{j,\ell}).$ Given $b^{(\ell)}_{\cdot}=1,$ $b^{(\ell)}_{0}$ becomes a simple mutlinomial distribution with joint probability mass function
$$
\indic_{\{b^{(\ell)}_{0}=1\}}\prod_{j=1}^{q}u^{b^{(\ell)}_{j}}_{j,\ell}
$$
where $u_{j,\ell}=r_{j,\ell}/\sum_{k=1}^{q}r_{k,\ell}.$
\end{enumerate}
\end{prop}
\begin{rem} We will discuss more details for the generative scheme following the next section which describes a general multivariate setting.
\end{rem}
\subsection{A general multivariate process}
Let now $A_{0}:=(A_{1},\ldots,A_{\upsilon}),$ for $\upsilon$ a positive integer not necessarily equal to $q,$  denote a random vector taking values in $\mathcal{A}\subseteq \mathbb{R}^{\upsilon},$ with conditional distribution $G_{A_{0}}(\cdot|\mathbf{s}_{q}),$ where $\mathbf{s}_{q}\in \mathbb{R}^{q}_{+},$ and such that $1-\pi_{A_{o}}(\mathbf{s}_{q})=\mathbb{P}(A_{1}=0,\ldots, A_{\upsilon}=0|\mathbf{s}_{q})>0.$
Then we can define a vector valued process ${Z}^{(i)}_{0}:=(Z^{(i)}_{1},\ldots,Z^{(i)}_{\upsilon}),$ where $Z^{(i)}_{j}=\sum_{k=1}^{\infty}A^{(i)}_{j,k}\delta_{\tilde{\omega}_{k}},$ and $Z^{(i)}_{\cdot}=\sum_{j=1}^{\upsilon}Z^{(i)}_{j}=\sum_{k=1}^{\infty}A^{(i)}_{\cdot,k}\delta_{\tilde{\omega}_{k}},$ where conditional on $\mu_{0}:=(\mu_{1},\ldots \mu_{\upsilon}),$ for each fixed $(i,k),$
$A^{(i)}_{0,k}:=(A^{(i)}_{1,k},\ldots, A^{(i)}_{\upsilon,k})$ is independent $G_{A_{0}}(\cdot|\mathbf{s}_{q,k}),$ where 
$\mathbf{s}_{q,k}=(s_{1,k},\ldots,s_{q,k})$ are vector valued points of a PRM with intensity $\rho_{q}(\cdot|\omega).$ We say that ${Z}^{(1)}_{0},\ldots,{Z}^{(M)}_{0}|\mu_{0}$ are iid $\mathrm{IBP}(G_{A_{0}}|\mu_{0}).$ It follows that, denoting the argument for $A^{(i)}_{0,j}$ as $\mathbf{a}^{(i)}_{0,j},$ the likelihood is given by, 

$$
{\mbox e}^{-\sum_{j=1}^{\infty}[-M\log(1-\pi_{A_{0}}(\mathbf{s}_{q,j}))]}\prod_{i=1}^{M}\prod_{j=1}^{\infty}{\left[\frac{G_{A_{0}}(d\mathbf{a}^{(i)}_{0,j}|\mathbf{s}_{q,j})}{
1-\pi_{A_{0}}(\mathbf{s}_{q,j})}\right]}^{\indic_{\{\mathbf{a}^{(i)}_{0,j}\notin {\mathbf{0}}\}}}
$$ 
Despite the extension to the multivariate case, it is evident that the form above looks quite similar to the univariate case and one may conclude that the relevant joint distribution of $N,(\mathbf{J}_{q,\ell}),(Z^{(1)}_{0},\ldots,Z^{(M)}_{0})$ is given by 
\begin{equation} 
\label{mainjoint3multi2}
\mathcal{P}(dN|\nu_{f_{M}},\mathbf{s,\omega}){\mbox e}^{-\Psi(f_{M})}
\prod_{\ell=1}^{K}
{[1-\pi_{A_{0}}(\mathbf{s}_{q,\ell})]}^{M}
\rho_{q}(d\mathbf{s}_{q,\ell}|\omega)B_{0}(d\omega_{\ell})
\prod_{i=1}^{M}h_{i,\ell}(\mathbf{s}_{q,\ell})
\end{equation}
where now $f_{M}(\mathbf{s}_{q},\omega)=-M\log[1-\pi_{A_{0}}(\mathbf{s}_{q})]$
and hence
$$
\nu_{f_{M}}(d\mathbf{s}_{q},\omega)={[1-\pi_{A_{0}}(\mathbf{s}_{q})]}^{M}
\rho_{q}(d\mathbf{s}_{q}|\omega)B_{0}(d\omega):=\rho_{q,M}(d\mathbf{s}_{q}|\omega)B_{0}(d\omega),$$
and similar to the univariate case, 
\begin{equation}
\label{hfunction2}
h_{i,\ell}(\mathbf{s}_{q})={\left[\frac{G_{A_{0}}(d\mathbf{a}^{(i)}_{0,\ell}|\mathbf{s}_{q})}{
1-\pi_{A_{0}}(\mathbf{s}_{q})}\right]}^{\indic_{\{\mathbf{a}^{(i)}_{0,\ell}\notin {\mathbf{0}}\}}},
\end{equation}
and
\begin{equation}
\Psi(f_{M})=\int_{\Omega}\int_{\mathbb{R}^{q}_{+}}(1-{[1-\pi_{A_{0}}(\mathbf{s}_{q})]}^{M})\rho_{q}(d\mathbf{s}_{q}|\omega)B_{0}(d\omega).
\label{psifunctionM2}
\end{equation}

\begin{prop} Suppose that ${Z}^{(1)}_{0},\ldots,{Z}^{(M)}_{0}|\mu_{0}$ are iid $\mathrm{IBP}(G_{A_{0}}|\mu_{0}),$ where $\mu_{0}$ is a multivariate $\mathrm{CRM}(\rho_{q}B_{0}).$ For each $A_{0},$ and $M$ set
$$
\rho_{q,M}(\mathbf{s}_{q}|\omega)={[1-\pi_{A_{0}}(\mathbf{s}_{q})]}^{M}\rho_{q}(\mathbf{s}_{q}|\omega).
$$
\begin{enumerate}
\item[(1)]
Then it follows that the posterior distribution of $N|(Z^{(1)}_{0},\ldots,Z^{(M)}_{0})$ is equivalent to the distribution of $N_{M}+\sum_{\ell=1}^{K}\delta_{\mathbf{J}_{q,\ell},\omega_{\ell}},$ where $N_{M}$ is a $\mathrm{PRM}(\rho_{q,M}B_{0})$ and the vector 
$\mathbf{J}_{q,\ell}=(J_{1,\ell},\ldots, J_{q,\ell})$ has joint density, with argument $\mathbf{s}_{q,\ell},$ given proportional to 
$$
{[1-\pi_{A_{0}}(\mathbf{s}_{q,\ell})]}^{M}
\rho_{q}(d\mathbf{s}_{q,\ell}|\omega)
\prod_{i=1}^{M}h_{i,\ell}(\mathbf{s}_{q,\ell}).$$
\item[(2)]Let $\mu_{0,M}=(\mu_{1,M},\ldots,\mu_{q,M})$ denote a multivariate $\mathrm{CRM}(\rho_{q,M}B_{0}), $ then the posterior distribution of $\mu_{0}|{Z}^{(1)}_{0},\ldots,{Z}^{(M)}_{0},$ is equivalent to that of a multivariate process whose $j$-th component is equivalent in distribution to 
$$
\mu_{j,M}+\sum_{\ell=1}^{K}J_{j,\ell}\delta_{\omega_{\ell}}.
$$
\item[(3)] The corresponding distribution of $ Z^{M+1}_{0}|(Z^{(1)}_{0},\ldots,Z^{(M)}_{0})$ can be represented in terms of   ${\tilde{Z}}^{(M+1)}_{0}:=(\tilde{Z}^{(M+1)}_{1},\ldots,\tilde{Z}^{(M+1)}_{\upsilon}),$ which is determined by $N_{M}$ call it an $\mathrm{IBP}(A
_{0},\rho_{q.M}B_{0})$ vector, and $\ell=1,\ldots, K$  vectors $(A^{(\ell)}_{0}), $ where for each fixed $\ell,$ $A^{(\ell)}_{0}|\mathbf{J}_{q,\ell}=\mathbf{s}_{q}$ has distribution $G_{A_{0}}(\cdot|\mathbf{s}_{q}).$ In other words component-wise $Z^{(M+1)}_{0}$
can be represented in distribution as $\tilde{Z}^{(M+1)}_{j}+\sum_{\ell=1}^{K}A^{(\ell)}_{j}\delta_{{\omega}_{\ell}},$ where $A^{(\ell)}_{j}$ is the j-th component of $A^{(\ell)}_{0}.$
\end{enumerate}
\end{prop}
It remains to sample $\tilde{Z}^{(M+1)}_{0},$ we close with a generalization of Proposition 2.2.

\begin{prop}\label{multisample}Let $Z_{0}=(Z_{1},\ldots,Z_{\upsilon})$ have a  $\mathrm{IBP}(A_{0},\rho_{q}B_{0})$ distribution,  where it is assumed that $\rho_{q}$ is homogeneous. If $$\varphi:=\varphi(\rho_{q}|\pi_{A_{0}})=\int_{\mathbb{R}^{q}_{+}}\pi_{A_{0}}(\mathbf{s}_{q})
\rho_{q}(d\mathbf{s}_{q})<\infty,$$ then for $j=1,\ldots \upsilon$
$$
Z_{j}\overset{d}=\sum_{k=1}^{\xi(\varphi)}X_{j,k}\delta_{\tilde{\omega}_{k}}
$$
where $X_{0,k}=(X_{1,k},\ldots X_{\upsilon,k})$ are iid across $k,$ and independent of $\xi(\varphi)$ a Poisson random variable with mean $\varphi.$ Furthermore setting 
$H_{0,k}=(H_{1,k},\ldots H_{q,k}),$ there are the iid pairs $((H_{0,k},X_{0,k}))$ 
with distributions
$$
\mathbb{P}(X_{0,k}\in d\mathbf{a}_{0}|H_{0,k}=\mathbf{s}_{q})=
{\left[\frac{\indic{\{\mathbf{a}_{0}\notin {\mathbf{0}}\}}G_{A_{0}}(d\mathbf{a}_{0}|\mathbf{s}_{q})}{
\pi_{A_{0}}(\mathbf{s}_{q})}\right]}
$$
and $\mathbb{P}(H_{0,k}\in d\mathbf{s}_{q})=\pi_{A_{0}}(\mathbf{s}_{q})\rho_{q}(d\mathbf{s}_{q})/\varphi(\rho_{q}|\pi_{A_{0}}).$
\end{prop}
\begin{proof}It follows that the distribution of $Z_{0}$ is characterized by the Laplace or characteristic functional of $\sum_{j=1}^{\upsilon}t_{j}Z_{j}(f),$ for general $t_{i}.$ Appealing again to the theory of Poisson marked processes one sees the L\'evy exponent is given by
$$
\int_{\Omega}\int_{\mathcal{A}}\int_{\mathbb{R}^{q}_{+}}(1-{\mbox e}^{-\sum_{j=1}^{\upsilon}t_{j}a_{j}f(\omega)})\indic{\{\mathbf{a}_{0}\notin {\mathbf{0}}\}}G_{A_{0}}(d\mathbf{a}_{0}|\mathbf{s}_{q})\rho_{q}(d\mathbf{s}_{q})B_{0}(d\omega).
$$
The condition $\varphi<\infty$ reduces the result to manipulation of the joint measure 
$$\indic{\{\mathbf{a}_{0}\notin {\mathbf{0}}\}}G_{A_{0}}(d\mathbf{a}_{0}|\mathbf{s}_{q})\rho_{q}(d\mathbf{s}_{q}).$$
\end{proof}
\begin{rem}Naturally for each $A_{0},$ and $M$ one uses Proposition~\ref{multisample} with 
$$
\rho_{q,M}(\mathbf{s}_{q})={[1-\pi_{A_{0}}(\mathbf{s}_{q})]}^{M}\rho_{q}(\mathbf{s}_{q})
$$
in order to sample  $\tilde{Z}^{(M+1)}_{0}.$
\end{rem}

\subsection{The multinomial case: Indian buffet process with a condiment}The simplest multivariate model is of course the multinomial case that was outlined in section ~\ref{Multi}.  Each non-zero entry in the matrix contains a vector of length $q$ where one entry takes the value $1$ and the other entries in the vector take the value $0.$ This could be described in terms of a customer selecting a certain dish but along with that choosing one particular condiment to go along with that dish. Perhaps a mango chutney or simply salt.  This means that 2 or more individuals may have selected the same dish (have the same basic trait) but might differ in terms of the accompanying condiment. Since $Z^{(i)}_{\cdot}=\sum_{j=1}^{q}Z^{(i)}_{j}$ is a Bernoulli process it is evident that customers choose basic dishes according to the Indian buffet process arising in the univariate case. In addition, given each new dish chosen, customer $M+1,$
selects with it the $j$-th of $q$ possible condiments with probability
$$
\int_{R^{q}_{+}}{p_{j}}\rho_{q,M}(dp_{1},\ldots,dp_{q})
$$
Customer $M+1$ also will possibly choose a previously selected dish $\omega_{\ell},$ with probability $\sum_{j=1}^{q}r_{j,\ell}$ and given this will choose one of q condiments, say j, (possibly different than what has previously been chosen by others) with probability $u_{j,\ell}.$ That is to say according to a Multinomial$(1,(r_{1,\ell},\ldots,r_{q,\ell}))$ distribution described in Proposition~\ref{multiprop}. 
The customer will do this for each of the $\ell=1,\ldots K$ previously selected dishes. 
\subsection{Multinomial Case: Stable-Beta-Dirichlet process priors}\label{bD}
One of the tasks in the mulivariate case is to find convenient priors for $\mu_{0}.$ In the simple multinomial case we now show how the class of Beta-Dirichlet priors introduced in Kim, James and Weissbach \cite{KimY2} leads to explicit results. We introduce a slight modification of this model which allows for power-law behavior in the sense of Teh and G\"orur\cite{TehG},  specify $\mu_{0}$ to be a stable-Beta-Dirichlet process with parameters $(\alpha,\beta+\alpha; \gamma_{1},\ldots,\gamma_{q};\theta)$ by setting
$$
\rho_{q}(p_{1},\ldots,p_{q})=\frac{\theta\Gamma(\sum_{j=1}^{q}\gamma_{j})}{\prod_{j=1}^{q}\Gamma(\gamma_{j})}
p^{-\alpha-\sum_{j=1}^{q}\gamma_{j}}_{\cdot}{(1-p_{\cdot})}^{\beta+\alpha-1}\prod_{j=1}^{q}p^{\gamma_{j}-1}_{j}\indic_{\{0<p_{\cdot}< 1\}},
$$ 
for $0\leq \alpha<1,$ $\beta>-\alpha$, $\theta>0$ and $\gamma_{j}>0$ for $j=1,\ldots q.$ When $\alpha=0,$ this is the Beta-Dirichlet process given in~\cite{KimY2}.
Making the change of variable $s=p_{\cdot},$ it is easy to check that $\mu_{\cdot}=\sum_{j=1}^{q}\mu_{j}$ is a stable-Beta process in the sense of \cite{TehG} with L\'evy density 
$\rho(s)=\theta s^{-\alpha-1}(1-s)^{\beta+\alpha-1}.$ It follows that 
for each $\ell,$ $\mathbf{J}_{q,\ell}=(J_{1,\ell},\ldots,J_{q,\ell})$ is such that $\sum_{j=1}^{q}J_{j,\ell}$ has a $\mathrm{Beta}(c_{\ell,M}-\alpha,M+\beta+\alpha-c_{\ell,M})$ distribution just as the univariate case. Furthermore $\mathrm{D}_{q,\ell}:=(D_{1,\ell},\ldots,D_{q,\ell}),$ where $D_{j,\ell}=J_{j,\ell}/\sum_{k=1}^{q}J_{k,\ell}$ is independent of $\sum_{k=1}^{q}J_{k,\ell}$ and is a $\mathrm{Dirichlet}(c_{1,\ell,M}+\gamma_{1},\ldots,c_{q,\ell,M}+\gamma_{q})$ vector. Note that $\mu_{0,M}$ is a stable-Beta-Dirichlet process with parameters $(\alpha,M+\beta+\alpha; \gamma_{1},\ldots,\gamma_{q};\theta).$

Hence customer $M+1$ chooses an existing dish $\omega_{\ell}$ and accompanying condiment $j$ with probability
$$
r_{j,\ell}=\frac{c_{j,\ell,M}+\gamma_{j}}{c_{\ell,M}+\sum_{k=1}^{q}\gamma_{k}}
\times \frac{c_{\ell,M-\alpha}}{M+\beta}.
$$
For $\tilde{Z}^{(M+1)}_{0}$ it follows that for $j=1,\ldots,q,$ 
$$
\tilde{Z}^{(M+1)}_{j}\overset{d}=\sum_{k=1}^{\xi(\varphi)}X_{j,k}\delta_{\tilde{\omega}_{k}}
$$
where
$$
\varphi=\theta\int_{0}^{1}s^{1-\alpha-1}(1-s)^{M+\beta+\alpha-1}ds=\frac{\theta\Gamma(1-\alpha)\Gamma(M+\beta+\alpha)}{\Gamma(M+\beta+1)}
$$
and for each $k,$ $X_{0,k}=(X_{1,k},\ldots, X_{q,k})$ is a simple multinomial with probability mass function
\begin{equation}
\indic_{\{\sum_{i=1}^{q}x_{i}=1\}}\prod_{j=1}^{q}{\left(\frac{\gamma_{j}}{\sum_{i=1}^{q}\gamma_{i}}\right)}^{x_{j}}.
\label{DirM}
\end{equation}
Thus customer $M+1$ chooses a Poisson $\varphi$ number of new dishes exactly as in the univariate case and also for each new dish chosen selects a single condiment $j$ with probability $\gamma_{j}/\sum_{i=1}^{q}\gamma_{i},$ for $j=1,\ldots,q.$
\begin{rem}Note that for each $(M,k),$ $H_{0,k}$ is marginally a Beta-Dirichlet random vector with parameters $(1-\alpha,M+\beta+\alpha; \gamma_{1},\ldots,\gamma_{q}).$ This means that $H_{\cdot,k}:=\sum_{j=1}^{q}H_{j,k}$ is a $\mathrm{Beta}(1-\alpha,M+\beta+\alpha)$ random variable and the vector $(D_{1,k},\ldots, D_{q,k}),$ for $D_{j,k}=H_{j,k}/H_{\cdot,k},$ is $\mathrm{Dirichlet}(\gamma_{1},\ldots,\gamma_{q}),$ which leads to (\ref{DirM}).
\end{rem}
\section{Modelling capabilities and challenges}
As discussed in~James~\cite{James2002, James2005} the PPC is a direct extension of a disintegration/Fubini calculus employed for gamma and Dirichlet processes by Albert Lo, which Lo further credits as techniques developed from conversations with Lucien Le Cam. The idea to extend this to a general Poisson random measure setting was suggested to this author by Jim Pitman in 2001.
The PPC seeks out the fundamental prior-posterior disintegration of the Bayesian joint distribution, and as demonstrated is designed to exploit the common features of random processes whose behavior is governed by a Poisson random measure. In other words it is a Palm calculus tailor-made for data structures arising in non-parametric Bayesian settings. This naturally includes any process based on completely random measures. This provides a unified approach for posterior analysis and also allows one to pinpoint differences between various processes. More general descriptions of posterior analysis are given in \cite{James2002,James2005}. Within our particular context, this analysis paves the way for the treatment of these infinite dimensional processes in a similar fashion to the now well understood Dirichlet process in complex applications. 

The remaining structure of the IBP and its generalizations mentioned here, as already evidenced by previous works, presents exciting opportunities for it seems an abundance of varied applications. We note the generalization to the multivariate setting accommodates practically any distribution. For example, it would be of interest to explore further the usage/interpretation of \emph{multivariate Bernoulli} distributions~\cite{Dai, Whitaker} within the IBP context. This class of distributions has been recently discussed in \cite{Dai} in relation to the the treatment of undirected graphical models with binary nodes as described in~\cite{Wainwright}. For illustration the bivariate Bernoulli distribution is defined in terms of random variables $A_{0}:=(A_{1},A_{2})$ taking values in the space ${\{0,1\}}^{2},$ with joint probability mass function, depending on probabilities $\mathbf{p}_{3}:=(p_{0,1},p_{1,0},p_{1,1}),$ with $1-\pi_{A_0}(\mathbf{p}
_{3}):=p_{0,0}=1-(p_{0,1}+p_{1,0}+p_{1,1}),$
$$
\mathbb{P}(A_{1}=a_{1},A_{2}=a_{2}|\mathbf{p}_{3})=p^{a_{1}a_{2}}_{1,1}
p^{a_{1}(1-a_{2})}_{1,0}p^{a_{2}(1-a_{1})}_{0,1}p^{(1-a_{1})(1-a_{2})}_{0,0}.
$$
This model fits into the framework of section 5, so it remains to obtain plausible/interesting interpretations and perhaps questions in regards to the choice of $\rho,$ for practical implementation. For general ideas in regards to construction of conjugate models see the work of~\cite{Orbanz2,Orbanz3} and \cite{Broderick4}.

\end{document}